\documentclass[12pt, twoside]{article}
\usepackage{amsfonts}
\usepackage{amsmath, amsthm, amscd, amsfonts, xpatch}
\usepackage{times}

\usepackage{enumitem,kantlipsum}
\setlist[enumerate]{itemsep=0mm}
\usepackage[english]{babel}
\usepackage[T1]{fontenc}
\usepackage[utf8]{inputenc}
\usepackage{bm}
\usepackage{ifxetex}
\usepackage{pb-diagram}
\usepackage{mathtools}
\usepackage{mathrsfs}
\usepackage{latexsym}
\usepackage{amssymb}
\usepackage{caption}

\usepackage{graphicx,float}
\usepackage{tikz}
\usepackage{pst-node}
\usepackage{tikz-cd}

\def\titlerunning#1{\gdef\titrun{#1}}
\makeatletter
\def\author#1{\gdef\autrun{\def\and{\unskip, }#1}\gdef\@author{#1}}
\def\address#1{{\def\and{\\\hspace*{18pt}}\renewcommand{\thefootnote}{}%
		\footnote {#1}}%
	\markboth{\autrun}{\titrun}}
\makeatother
\def\email#1{e-mail: #1}

\def\keywords#1{\par\medskip
	\noindent\textbf{Keywords.} #1}
\def\MSC#1{\par\medskip
	\noindent\textbf{MSC.} #1}
\usepackage{lineno}
\usepackage{csquotes}

\usepackage[a4paper,inner=2cm,outer=2cm,top=2cm,bottom=2cm,footskip=1cm]{geometry}

\definecolor{persianblue}{rgb}{0.11, 0.22, 0.73}
\definecolor{persiangreen}{rgb}{0.0, 0.65, 0.58}

\usepackage{url}

\usepackage{nameref}
\usepackage[hidelinks]{hyperref}
\hypersetup{
	colorlinks   = true, 
	urlcolor     = persiangreen, 
	linkcolor    = persianblue, 
	citecolor   = persianblue 
}
\usepackage[capitalize]{cleveref}
\Crefname{theorem}{Theorem}{Theorems}
\Crefname{lemma}{Lemma}{Lemmas}
\crefname{claim}{Claim}{Claims}

\usepackage[defernumbers=true,maxbibnames=6,backend=biber,style=numeric,giveninits=true,sortcites=true,doi=false,isbn=false,url=false,safeinputenc]{biblatex}
\usepackage[nottoc,numbib]{tocbibind} 

\newtheorem{theorem}{Theorem}[section]
\newtheorem*{theorem*}{Theorem}
\newtheorem*{mainthm*}{Main Theorem}

\newtheorem{lemma}[theorem]{Lemma}

\newtheorem{proposition}[theorem]{Proposition}
\newtheorem{corollary}[theorem]{Corollary}

\newtheorem{problem*}{Problem}[]

\newtheorem{fact}[theorem]{Fact}
\newtheorem{definition}[theorem]{Definition}

\newtheorem{notation}[theorem]{Notation}
\newtheorem{claim}[theorem]{Claim}
\numberwithin{equation}{section}
\newtheorem{convention}[theorem]{Convention}

\newcommand{\dom}[1]{{\rm dom}(#1)}

\newcommand{\rest}{\! \upharpoonright \!}

\newcommand{\forces}{\Vdash}

\newtheorem{imp-remark}[theorem]{\textbf{Important Remark}}

\numberwithin{equation}{section}

\theoremstyle{remark}
\newtheorem{remark}[theorem]{Remark}
\makeatletter
\let\qed@empty\openbox 
\def\@begintheorem#1#2[#3]{%
	\deferred@thm@head{%
		\the\thm@headfont\thm@indent
		\@ifempty{#1}
		{\let\thmname\@gobble}
		{\let\thmname\@iden}%
		\@ifempty{#2}
		{\let\thmnumber\@gobble\global\let\qed@current\qed@empty}
		{\let\thmnumber\@iden\xdef\qed@current{#2}}%
		\@ifempty{#3}
		{\let\thmnote\@gobble}
		{\let\thmnote\@iden}%
		\thm@swap\swappedhead
		\thmhead{#1}{#2}{#3}%
		\the\thm@headpunct\thmheadnl\hskip\thm@headsep
	}\ignorespaces
}
\renewcommand{\qedsymbol}{%
	\ifx\qed@thiscurrent\qed@empty
	\qed@empty
	\else
	\fbox{\scriptsize\qed@thiscurrent}%
	\fi
}
\renewcommand{\proofname}{%
	Proof%
	\ifx\qed@thiscurrent\qed@empty
	\else
	\ of \qed@thiscurrent
	\fi
}
\xpretocmd{\proof}{\let\qed@thiscurrent\qed@current}{}{}
\newenvironment{proof*}[1]
{\def\qed@thiscurrent{\ref{#1}}\proof}
{\endproof}
\makeatother

\makeatletter
\let\blx@rerun@biber\relax
\makeatother

\begin{document}
	
	
	\baselineskip=17pt
	
	

	\titlerunning{Specialising Trees With Small Approximations I}
	\title{Specialising Trees With Small Approximations I}

	\author{Rahman Mohammadpour}
	
	\date{}
	
	\maketitle
	
	\address{Rahman Mohammadpour: Institut für Diskrete Mathematik und Geometrie, TU Wien,
		1040 Vienna, Austria. \email{rahmanmohammadpour@gmail.com}}


\begin{abstract}
Assuming $\rm PFA$,  we shall use internally club $\omega_1$-guessing models as side conditions to show that for every tree $T$  of height $\omega_2$ without cofinal branches,  there is a proper and $\aleph_2$-preserving forcing notion with finite conditions which specialises $T$. Moreover, the forcing has the $\omega_1$-approximation property.

 \MSC{03E05, 03E35, 03E57} 
\end{abstract}


\section{Introduction}

By the well-known work of Baumgartner, Malitz and Reinhardt \cite{BaumMalRei},  under Martin's axiom at $\aleph_1$, all trees of height and size $\omega_1$ without cofinal branches are special. Unfortunately, the straightforward generalisations of $\rm MA$ were not thus far capable of specialising $\omega_2$-Aronszajn trees, see \cite{Baumgartnersurvey,Shelah-wMA,Shelah-Stanley-GenMA}. The different behaviour of the specialising problem  beyond $\omega_1$ arises from two interconnected factors:  the weakness of the current technology of forcing iterations and the nature of trees of height at least $\omega_2$.
Thus, the question of finding a legitimate higher version of Martin's axiom, under which every $\omega_2$-Aronszajn tree is special seems challenging (we will say  more about this.)
However, there are still many intriguing results in this research direction. For example, Laver and Shelah \cite{LavShe} showed, assuming the consistency of a weakly compact cardinal, that the $\omega_2$-Suslin Hypothesis is consistent with the Continuum Hypothesis (in fact, they showed that it is consistent that there are $\omega_2$-Aronszajn trees and all of them are special.) This result was extended by  Golshani and Hayut in \cite{GolHay}, where they proved that, modulo the consistency of large cardinals,  it is consistent that for every regular cardinal $\kappa$, there are $\kappa^{+}$-Aronszajn trees and all of them are special. A more relevant result, where wide trees were involved, was obtained by  Golshani and Shelah in \cite{GolShe}, that is for a prescribed regular cardinal $\kappa$, it is consistent that every tree of height and size $\kappa^+$ (with a small number of branches),  is weakly special (i.e., there is a colouring with $\kappa$ colours so that if $s<t,u$ have the same colour, then $t$ and $u$ are comparable.)
The affinity between these and other similar results is that they rely upon the original technique of Laver and Shelah \cite{LavShe}. Although the main difficulty in proving an iteration theorem for countably closed and $\aleph_2$-c.c forcings is the preservation of  $\aleph_2$,  it was  surmountable by Laver-Shelah's argument due to the particular features of the iterands.
The attempts to overcome the difficulty and  find a higher analogue of $\rm MA$ have been  generally devoted to countably closed forcings until Neeman's discovery \cite{NE2014}  of generalised side conditions. His technology allows us to examine the connection between the specialisation problem and generalised forms of Martin's axiom, and ask if we still need to consider countably closed forcings in this context.
If the consistency of a higher analogue of $\rm PFA$ is achievable,  it is then natural to speculate whether such a forcing axiom can imply that all trees in an appropriate subclass of trees of height and size $\omega_2$ are special. 
As an early application of his method, Neeman \cite{Neeman2017} attempted to (partially) specialise trees of height $\omega_2$ with finite conditions. To achieve this, he attaches the partial specialising  
functions to the sequences of models as side conditions. He then demonstrates that the resulting construction belongs to an iterable class which also includes a forcing notion for adding a nonspecial $\omega_2$-Aronszajn tree.

The second factor mentioned above may also  lead one to recast the program of finding a generalised $\rm MA$ for the problem of special $\omega_2$-Aronszajn trees, as such trees intrinsically involve a particular compactness phenomenon.
One can use some forms of the square principle to construct trees without cofinal branches that cannot be special, even in transitive outer  models with the same cardinals.  The basic idea goes back to Laver (see \cite{ShelahStanley}) who isolated  the concept of an ascending path through a tree and showed that an $\omega_2$-Aronszajn tree with an ascending path is non-special even in any transitive outer model that computes the relevant cardinals correctly. However, the earliest example of a non-special $\omega_2$-Aronszajn tree was constructed by Baumgartner using $\square_{\omega_1}$, which was also independently discovered and generalised by Shelah and Stanley \cite{ShelahStanley}.
They showed that $\square_{\lambda}$ implies the existence of non-specialisable $\lambda^+$-Aronszajn trees. The connection between  square-like  principles and ascending paths through trees or tree-like systems has been studied by several people, just to mention a few:
Baumgartner (as mentioned above),
Brodsky and Rinot \cite{BrodskyRinot},
Devlin \cite{Devlin83},
Cummings \cite{Cummings97},
 Lambie-Hanson\cite{Hanson-Sq-Narrow}, Lamibie-Hanson and  L\"{u}cke \cite{, LambieLucke}, 
 Laver and Shelah \cite{LavShe},
 L\"{u}cke \cite{Lucke},
 Neeman \cite{Neeman2017},
Shelah and Stanley  \cite{ShelahStanley},
Todorčević \cite{Todorcevic-SpecialSquare}.

To see why specialising a tree of height beyond $\omega_1$ is subtly different from that of a tree of height $\omega_1$, let us first recall that the standard forcing to specialises a tree $T$ of height $\kappa^+$ uses partial specialising functions of size less than $\kappa$, and let us denote this forcing by $\mathbb S_{\kappa}(T)$. For a cardinal $\lambda\leq\kappa$, $\mathbb S_\lambda(T)$ is defined naturally.  L\"{u}cke \cite{Lucke} studied the chain condition of $\mathbb S_\lambda(T)$,  and complete the bridge between the notion of an ascending path and the chain condition of  $\mathbb S_\lambda(T)$. Under some cardinal arithmetic assumptions, he showed that the nonexistence of a \emph{weak form of ascending paths}\footnote{ See \cite{Lucke} for the definition.}  of width less than $\lambda$ through $T$ is equivalent to the $\kappa^+$-chain condition of $\mathbb S_{\lambda}(T)$.
Note that it is easily seen that $\mathbb S_{\lambda}(T)$ collapses $\kappa^+$ if $T$ has a cofinal branch.
Observe that also by Baumgartner--Malitz--Reinhardt \cite{BaumMalRei}, if $T$ is of height $\omega_1$ without cofinal branches, then  $\mathbb S_{\omega}(T)$ has the countable chain condition, as the existence of a cofinal branch through such tree is equivalent to the existence of a (weak) ascending path of finite length.  
It is also not hard to see that if $\kappa=\omega_1$ and the $\rm CH$ fails, then $\mathbb S_{\omega_1}(T)$ collapses the continuum onto $\omega_1$. Thus not only the  $\rm CH$ is necessary for preserving $\aleph_2$, but  also by L\"{u}cke's result, the lack of cofinal branches through $T$ is not enough to ensure that $\mathbb S_{\omega_1}(T)$ preserves $\aleph_2$. On the other hand, if $T$ is of height $\omega_2$ and has no cofinal branches, then $\mathbb S_{\omega}(T)$ has the $\aleph_2$-chain condition, but then the question is how to preserve $\omega_1$?

Therefore,  the behaviour of the continuum function and the existence of ascending paths of width $\omega$ can prevent us from specialising trees of height $\omega_2$ merely with countable conditions. L\"{u}cke \cite{Lucke} asked  the following questions:
\begin{enumerate}
  
    \item Assume $\rm PFA$. Is  every tree of height $\omega_2$ without cofinal branches specialisable?
     \item If $T$ is a tree of height $\kappa^+$, for an uncountable regular cardinal $\kappa$ without ascending paths of width less than $\kappa$, is then $T$ specialisable?
\end{enumerate}

Let us end our discussion with a couple of  general questions: Do we still need to consider the specialisation of all $\omega_2$-Aronszajn trees in the context of generalised Martin's axiom? If looking for a generalised $\rm MA$, do we want to have some kinds of compactness at $\aleph_2$ or not?

In this paper, we prove the following theorem.
\begin{theorem*}
Assume $\rm PFA$. Every tree of height $\omega_2$  without cofinal branches is specialisable via a proper and $\aleph_2$-preserving forcing with finite conditions. Moreover, the forcing has the $\omega_1$-approximation property. 
\end{theorem*}

This theorem  answers L\"{u}cke's first question in the affirmative.\footnote{To be precise,  L\"{u}cke's definition of the specialisability of a tree $T$ requires preservation of all cardinal up to the size of $T$, however in our theorem the size of $T$ will be collapsed to $\aleph_2$.}
Given a tree $T$ of height $\omega_2$ with no cofinal branches,
we shall use internally club $\omega_1$-guessing models to construct a proper forcing notion $\mathbb P_T$ similar to Neeman's in \cite{Neeman2017},
so that forcing with $\mathbb P_T$ specialises $T$.
Notice that the existence of sufficiently many $\omega_1$-guessing models of size $\omega_1$ implies the failure of certain versions of the square principle.  It is also worth mentioning that by an observation due to L\"{u}cke,  the existence of sufficiently many $\omega_1$-guessing models of size $\aleph_1$ (and hence under $\rm PFA$) no tree of height $\omega_2$ without cofinal branches contains an ascending path of width $\omega$. Interestingly, we will not use this fact, as the presence of guessing models in our side conditions suffices. By a theorem due to Viale and Wei\ss~\cite{VW2011}, under $\rm PFA$, there are stationarily many internally club guessing models, and by a theorem due to Cox and Krueger~\cite{CK-Quotient}, this consequence of $\rm PFA$  is consistent with arbitrarily large continuum. Thus essentially, the fact that $2^{\aleph_0}=\aleph_2$ holds under $\rm PFA$ does not play a role in our result and proofs.

We shall also answer the second question above consistently in the affirmative, for trees of height $\kappa^{++}$  without cofinal branches, in our forthcoming paper \cite{M-in-prep}, which in particular includes a proof of the following theorem.

\begin{theorem*}[\cite{M-in-prep}]
Assume  $\kappa$ is a regular cardinal, and that $\lambda>\kappa$ is a supercompact cardinal.
Then in generic extensions by some $<\!\kappa$-closed forcing notion, 
$\kappa^{<\kappa}=\kappa$, $2^{\kappa}=\lambda=\kappa^{++}$ and every tree of height $\kappa^{++}$ without cofinal branches is specialisable via some $<\!\kappa$-closed forcing which preserves $\kappa^+$ and $\kappa^{++}$.
\end{theorem*}

Our paper includes four additional sections. We give the preliminaries in \cref{sec2}. \cref{sec3} is devoted to the introduction and the basic properties of forcing with pure side conditions. We shall introduce our main forcing and state its basic properties in \cref{sec4}. Finally, we establish our main result in \cref{sec5}.

\section{Preliminaries}\label{sec2}

We shall follow  standard conventions and notation, but let us recall some of the most important ones.
 In this paper,
 by $p\leq q$ in a forcing ordering $\leq$, we mean $p$ is stronger than $q$;
    for a cardinal $\theta$, $H_\theta$ denotes the collection of sets whose hereditary size is less than $\theta$; 
 for a set $X$, we let $\mathcal P(X)$ denote the power-set of $X$, and if  $\kappa$ is a cardinal, we let $\mathcal P_\kappa(X)\coloneqq\{A\in\mathcal P(X):|A|<\kappa\}$;
 recall that a set $\mathcal S\subseteq \mathcal P_\kappa(H_\theta)$ is  stationary, if for every function $F:\mathcal P_{\aleph_0}(H_\theta)\rightarrow \mathcal P_\kappa(H_\theta)$, there is $M\prec H_\theta$  in $\mathcal S$ with $M\cap\kappa\in \kappa$ such that $M$ is closed under $F$.

\subsection{Trees}\leavevmode

Let us  recall the definition of a tree and some related concepts.

\begin{definition}
A \emph{tree} is a partially ordered set $(T,<_T)$ such that for every $t\in T$, $b_t\coloneqq\{s\in T: s<_T t\}$ is well ordered with respect to $<_T$.
\end{definition}
\begin{definition}
  Suppose $T=(T,<_T)$ is a tree.
 \begin{enumerate}
   \item For every $t\in T$, the \emph{height} of $t$, denoted by ${\rm ht}_T(t)$, is the order type of $b_t$.
    \item The \emph{height} of $T$, denoted by ${\rm ht}(T)$, is ${\rm sup}\{{\rm ht}_T(t)+1:t\in T\}$.

     \item For every $\alpha\leq {\rm ht}(T)$, $T_{\alpha}$ denotes the set of nodes of height $\alpha$. $T_{\leq\alpha}$ and $T_{<\alpha}$ have the obvious meanings.  In particular, $T=T_{<{\rm ht}(T)}$ and $T_{{\rm ht}(T)}=\varnothing$.
  
     \item A set $b\subseteq T$ is called a \emph{branch} through $T$ if $(b,<_T)$ is a downward-closed and linearly ordered set. A branch is  a \emph{cofinal branch} if its order type is the height of $T$.
    
     \item $T$ is called \emph{Hausdorff} if for every limit ordinal $\alpha$ ($\alpha=0$ is allowed), and every $t\neq s$ in $T_\alpha$, we have $b_t\neq b_s$.
   
      \item For every $t\in T$,  we let  $\overline{b}_t$ denotes $\{s\in T: s\leq_T t\}$.
     \end{enumerate}
\end{definition}

Observe that a Hausdorff tree is rooted, i.e., it has a unique minimal point.

\begin{definition}
Suppose  $\kappa$ is an infinite cardinal.
  A tree $(T,<_T)$ of height $\kappa^+$ is called \emph{special} if there is a  \emph{specialising function}  $f:T\rightarrow\kappa$, i.e.,  if $s<_T t$, then 
  $f(s)\neq f(t)$. 
\end{definition}

\begin{definition}\label{B forcing}
Suppose that $\lambda\leq \kappa$ are infinite regular cardinals.
 Assume that $T$ is  a tree of height $\kappa^+$. Let $\mathbb S_{\lambda}(T)$ denote the forcing notion consisting  of  partial specialising functions, of size less than $\lambda$,  ordered by reversed inclusion, that is $f\in\mathbb S_{\lambda}(T)$ is a partial function from $T$ to $\kappa$ such that if $s,t\in{\rm dom}(f)$ are comparable in $T$, then $f(t)\neq f(s)$.
\end{definition}

\begin{lemma}\label{tree iso}
In order to specialise a tree $T$ (of height $\kappa^+$, for some infinite cardinal $\kappa$), one may assume, without loss of generality, that $T$  is a Hausdorff tree.
\end{lemma}
\begin{proof}
Recall that a function $f:T_1\rightarrow T_2$ between two trees is called a weak embedding if $f$ respects the strict orders.  It is easily seen that if $T_1$ weakly embeds into $T_2$ and $T_2$ is special, then $T_1$ is special, as the inverse image of an antichain in $T_2$ under a weak embedding is an antichain in $T_1$.
Thus to prove the lemma, it is enough to show that there is a weak embedding from $T$ into a Hausdorff tree $T^*$ of the same height as $T$. 

Let $T^*$ be the set of all non cofinal  branches through $T$. Then, $(T^*,\subset)$ is a  tree of the same height as $T$. Note that $\varnothing $ is the root of $T^*$. Moreover, if $a\in T^*$, then the order type of $(a,<_T)$ is exactly ${\rm ht}_{T^*}(a)$. Suppose that $\alpha$ is a nonzero limit ordinal and $a, a'\in T^*_\alpha$ with $b_a=b_{a'}$. We claim that $a=a'$. Let $t\in a$. Since the order type of $a$ is a limit ordinal, there is $s\in a$ with $t<_T s$. Let $x=\{u\in T: u<_T s\}$. Now $x <_{T^*}a$. Thus $x\in b_{a}=b_{a'}$. Then  $t\in x\subseteq a'$. So we have $a\subseteq a'$. Similarly, we have $a'\subseteq a$, and therefore, $a=a'$. 
Now, let $f:T\rightarrow T^*$ be defined by $f(t)=b_t$. If $s<t$, then $b_s$ is a proper subset of $b_t$, and hence $f$ is a weak embedding.

\end{proof}

\subsection{Strong Properness and the Approximation Property}\leavevmode

Recall that if $M\prec H_\theta$ contains a forcing $\mathbb P$, then a condition $p\in\mathbb P$  is called $(M,\mathbb P)$-generic if for every dense subset $D$ of $\mathbb P$ in $M$, $D\cap M$ is pre-dense below $p$. 

\begin{definition}
Assume that $\mathbb P$ is a forcing, and $\theta$ is a sufficiently large regular cardinal.
Suppose $\mathcal S\subseteq\mathcal P_\kappa(H_\theta)$ consists of elementary submodels.
Then, $\mathbb P$ is said to be \emph{proper for $\mathcal S$}, if for every $M\in\mathcal S$ and every $p\in \mathbb P\cap M$, there is an $(M,\mathbb P)$-generic condition $q\leq p$ .
\end{definition}

\begin{lemma}\label{preserv-byS-proper}
Let $\kappa$ be a regular cardinal.
Assume that $\mathbb P$ is a forcing, and $\theta>\kappa$ is a sufficiently large regular cardinal. Suppose $\mathcal S\subseteq\mathcal P_\kappa(H_\theta)$ is a stationary set of elementary submodels.  If $\mathbb P$ is proper for $\mathcal S$, then $\mathbb P$ preserves the regularity of $\kappa$.
\end{lemma}
\begin{proof}
Let $\gamma<\kappa$ be an ordinal. Assume towards a contraction that some $p\in\mathbb P$ forces  that $\dot{f}$ is an unbounded function from $\gamma$ into $\kappa$.  Pick $M\in\mathcal S$ such that $\gamma,\kappa,p,\dot{f}\in M$.
Let $q\leq p$ be an $(M,\mathbb P)$-generic condition.
Note that $\gamma\subseteq M$ and $M\cap\kappa\in\kappa$.
By our assumption, we can find a condition $q'\leq q$, and ordinal $\zeta<\gamma$ and  an ordinal $\delta\geq M\cap\kappa$ such that, $q'\Vdash `` \dot{f}(\zeta)=\delta "$.
 Set 
$$D=\{r\leq p: r \mbox{ decides the value } \dot{f}(\zeta)\}\cup\{r\in\mathbb P:r\perp p\}.$$
Then $D$ is a dense subset of $\mathbb P$ and belongs to $M$.
Since $q$ is $(M,\mathbb P)$-generic, there is $r\in D\cap M$ such that $r|| q'$. Thus $r$ is compatible with $p$, and hence, by elementarity, there is $\delta'\in M$ such that $r\Vdash ``\delta'= \dot{f}(\zeta)"$.
Now if $s$ is a common extension of $q'$ and $r$, we have 
$s\Vdash ``\delta'=\delta"$. Thus $\delta'=\delta\in M\cap\kappa$, a contradiction!

\end{proof}

Let us now recall  the following closely related definitions from \cite{MI2005} and  \cite{Hamkins2001}, respectively.
\begin{definition}[strong properness]
  Suppose $\mathbb P$ is a forcing notion.
  \begin{enumerate}
      \item Let $X$ be a set. A condition $p\in\mathbb P$ is said to be \emph{strongly $(X,\mathbb P)$-generic}, if for every $q\leq p$, there is some $q\rest_X\in X\cap\mathbb P$ such that every condition $r\in \mathbb P\cap X$ extending $q\rest_X$ is compatible with $q$.
      \item  For a collection of sets  $\mathcal S$, we say $\mathbb P$ is  \emph{strongly proper for $\mathcal S$}, if for every $X\in\mathcal S$ and every $p\in\mathbb P\cap X$, there is a strongly $(X,\mathbb P)$-generic condition extending $p$.
  \end{enumerate}
\end{definition}

\begin{remark}\label{remark-st-prop-preser}
It is easily seen that if $p$ is strongly $(X,\mathbb P)$-generic and $M\prec H_\theta$ is such that $M\cap \mathbb P=X\cap\mathbb P$, then $p$ is
strongly $(M,\mathbb P)$-generic, and hence $(M,\mathbb P)$-generic.
It turns out that if a forcing notion is strongly proper for  some stationary set $\mathcal S\subseteq\mathcal P_{\kappa}(H_\theta)$, then $\mathbb P$ is $\mathcal S$-proper, and hence it preserves $\kappa$,
 by \cref{preserv-byS-proper}.
 \end{remark}

\begin{definition}[$\kappa$-approximation property]
 Suppose $\kappa$ is an uncountable regular cardinal.
 A forcing notion $\mathbb P$ has the \emph{$\kappa$-approximation property}, if for every $V$-generic filter $G$, and every $A\in V[G]$ with $A\subseteq V$, the following are equivalent.
\begin{enumerate}
\item $A\in V$.
\item 
For every $a\in V$  with $|a|^V<\kappa$, we have $a\cap A\in V$.
\end{enumerate} 
\end{definition}

Note that it is well-known that if a  forcing notion is strongly proper for sufficiently many models in $\mathcal P_\kappa(H_\theta)$, then it has the $\kappa$-approximation property, see \cite{Mitchell-On-the-Hamkins}.

\subsection{Guessing Models}\leavevmode

For a set $M$, we say that a set $x\subseteq M$ is
{\em bounded in} $M$ if there is $y\in M$ such that $x\subseteq y$.
Recall that  an  elementary submodel $M$ of $H_\theta$ is called an \emph{internally club model} (or  \emph{$\rm IC$-model} for short) if it is the union of a continuous $\in$-sequence $\langle M_\alpha:\alpha<\omega_1\rangle$ of countable elementary submodels of $H_\theta$.

\begin{notation}
For   a model $M\prec H_\theta$, let $\kappa_M=\min\{\alpha\in M\cap\theta: \alpha\nsubseteq M\}$. Let $\kappa_M$ be undefined if the above supremum does not exist.
\end{notation}

\begin{definition}
  Suppose $M$ is a set. A set $x$ is \emph{guessed} in $M$ if there is some $x^*\in M$ such that $x^*\cap M=x\cap M$.
\end{definition}

We now recall the definition of a guessing model from \cite{VW2011}.

\begin{definition}[$\gamma$-guessing model]\label{guessing-def}
Assume that $\theta$ is an uncountable regular cardinal. Let $M\prec H_\theta$.
Suppose that $\gamma\in M$  is a regular cardinal with $\gamma\leq\kappa_M$. 
Then  $M$ is said to be a $\gamma$-{\em guessing model} if the following are equivalent for any  $x$ which is bounded in  $M$.
\begin{enumerate}
\item $x$ is $\gamma$-approximated in $M$, i.e., $x\cap a \in M$, for all $a\in M$ of size less than $\gamma$.
\item $x$ is guessed in $M$.
\end{enumerate}

\end{definition}

\begin{definition}[${\rm GM}^*(\omega_2)$]
The principle ${\rm GM}^*(\omega_2)$ states that for every sufficiently large regular cardinal $\theta$, the set of $\omega_1$-guessing elementary  $\rm IC$-submodels of $H_\theta$   is stationary in $\mathcal P_{\omega_2}(H_\theta)$.
\end{definition}

The above principle is  slightly stronger than Wei\ss 's $\rm ISP(\omega_2)$, see \cite{W2010,W2012} for more information on  $\rm ISP(\omega_2)$, which is also equivalent to the principle ${\rm GM}(\omega_2)$ that states  for every sufficiently large regular cardinal $\theta$, the set of $\omega_1$-guessing  elementary submodels of $H_\theta$   is stationary in $\mathcal P_{\omega_2}(H_\theta)$.

\begin{proposition}[Viale--Wei\ss, \cite{VW2011}]\label{fact Viale-Weiss}
$\rm PFA$ implies ${\rm GM}^*(\omega_2)$.
\end{proposition}
\begin{proof}
The proposition above was mentioned without proof in \cite{VW2011}. A sketch of a proof can be found in \cite[Theorem 4.4]{V2012}.

\end{proof}

The following lemma  plays a crucial role in our later proofs.

\begin{lemma}\label{Baumgartner}
Suppose $\theta$ is an uncountable regular cardinal.
Assume  that $M\prec H_\theta$  is countable.
Let $Z\in M$  a set.
Suppose that $z\mapsto f_z$ is a function on $\mathcal P_{\omega_1}(Z)$ in $M$, where for each $z\in\mathcal P_{\omega_1}(Z)$, $f_z$ is a $\{0,1\}$-valued function with $z\subseteq {\rm dom}(f_z)$.
Assume that $f:Z\cap M\rightarrow 2$  is a function that is not guessed in $M$.
Suppose that $B\in M$ is a  cofinal subset of $\mathcal P_{\omega_1}(Z)$.
Then there is $B^*\in M$ cofinal in  $B$
  such that for every $z\in B^*$, $f_z\nsubseteq f$.
\end{lemma}
\begin{proof}
For each $\zeta\in Z$, and $\epsilon=0,1$,  let
$$A_\zeta^\epsilon=\{z\in B: \zeta\in\dom(f_z) \mbox{ and }f_z(\zeta)=\epsilon\}.$$
Notice that the sequence 
$$\langle A_\zeta^\epsilon:\zeta\in Z,\epsilon\in\{0,1\}\rangle$$
belongs to $M$.
We are done if there is some $\zeta\in Z$ such that both $A_\zeta^0$ and $A_\zeta^1$ are  cofinal in $B$, as then by elementarity one can find such $\zeta\in  M\cap Z$, and then pick $A_\zeta^{1-f(\zeta)}$.
 Therefore, let us assume that  for every $\zeta\in Z$, there is an $\epsilon\in\{0,1\}$, which is necessarily unique,
such that  $A_\zeta^\epsilon$ is cofinal in $B$.
Now, define  $h$ on $ Z$ by letting $h(\zeta)$ be $\epsilon$ if and only if 
$A_\zeta^\epsilon$ is cofinal is $B$. Clearly  $h$ is in $M$, but then
$h\rest_{M}\neq f$
since $f$ is not guessed in $M$. Thus, there exists $\zeta\in  M\cap Z$ such
that $h(\zeta)\neq f(\zeta)$, but  it then implies that $A_\zeta^{1-f(\zeta)}$ is  cofinal in $B$ and  belongs to $M$.
Let $B^*$ be $A_\zeta^{1-f(\zeta)}$. Now if $z\in B^*$, $f_z\nsubseteq f$.
\end{proof}

\section{Pure Side Conditions}\label{sec3}
This section is devoted to the forcing with pure side conditions.  Such a forcing notion, as well as a finite-support iteration of proper forcings with side conditions, was introduced by Neeman in \cite{NE2014}.  However, we cannot use Neeman forcing directly,  since we shall work with non-transitive models. Instead, we follow Veličković's presentation \cite{VeliVen} of Neeman forcing with finite $\in$-chains of models of two types, where both types of models are non-transitive.  We shall sketch some proofs of the necessary facts in this section, and we encourage the reader to consult \cite{VeliVen} for more details.

Fix an  uncountable regular cardinal $\theta$, and let $x\in H_\theta$ be arbitrary. We let $\mathcal{E}^0\coloneqq\mathcal E^0(x)$ denote the collection of all countable elementary  submodels of $(H_{\theta},\in,x)$, and let $\mathcal{E}^1\coloneqq\mathcal E^1(x)$ denote a collection of  elementary $\rm IC$-submodels of  $(H_{\theta},\in,x)$. Note that for every $N\in\mathcal E^1$ and every $M\in\mathcal E^0$, if $N\in M$, then $N\cap M\in\mathcal E^0\cap N$.

\begin{definition}
 Assume that $\mathcal M\subseteq\mathcal{E}^0\cup\mathcal{E}^1$.
  
  \begin{enumerate}
  \item Suppose that $M,N\in\mathcal M$. We say $M$ is below $N$ in $\mathcal M$, or equivalently $N$ is above $M$ in $\mathcal M$, and denote this by $M\in^* N$ if there is a finite set $\{M_i:i\leq n\}\subseteq\mathcal M$ such that
  $M=M_0\in\dots \in M_n=N$.
      \item We say $\mathcal M$ is an $\in$-chain, if for every distinct $M,N\in\mathcal M$, either $M\in^* N$ in $\mathcal M$ or $N\in^* M$ in $\mathcal M$.
      \item We say $\mathcal M$ is closed under intersections if for every $M\in\mathcal M\cap\mathcal{E}^0 $, and every $N\in M\cap \mathcal M$, $N\cap M$ belongs to $\mathcal M$.
      \item If $M,N\in\mathcal M\cup\{\varnothing, H_\theta\}$, then by $(M,N)_{\mathcal{M}}$, and  intervals of other types, we mean that the interval is considered in the linearly ordered structure $(\mathcal M,\in^*)$, e.g., $(M,N)_{\mathcal M}=\{P\in\mathcal M: M\in^* P\in^* N\}$.  
      
  \end{enumerate}
\end{definition}

It is easily seen that if $M\in^* N$ holds in an $\in$-chain $\mathcal M$, and that $N\in\mathcal E^1$, then $M\in N$. We simply write $M\in^* N$, whenever $\mathcal M$ is clear from the context.

\begin{remark}\label{smallremark}
If $M,N\in\mathcal E^0$, then $M\subseteq N$ if and only if there is no $P\in\mathcal E^1\cap\mathcal M$ with $P\cap N \in ^* M\in^* P\in N$. 
\end{remark}

\begin{definition}[forcing with pure side conditions]
  We let $\mathbb M(\mathcal{E}^0,\mathcal{E}^1)$ denote the collection of  $\in$-chains $p=\mathcal M_p\subseteq \mathcal{E}^0\cup\mathcal{E}^1$ which are closed under intersections.
  We consider $\mathbb M(\mathcal{E}^0,\mathcal{E}^1)$ as a notion of forcing ordered by reversed inclusion.
\end{definition}

We  simply  denote $\mathbb M(\mathcal{E}^0,\mathcal{E}^1)$ by $\mathbb M$ whenever there are no confusions.
For a condition $p\in\mathbb M$, we let also $\mathcal{E}^0_p$ and $\mathcal{E}^1_p$ denote $\mathcal M_p\cap\mathcal{E}^0$ and $\mathcal M_p\cap \mathcal{E}^1$, respectively. If $p=(\mathcal M_p,\dots)$ is a condition  in a forcing notion with $\mathcal M_p\in\mathbb M$, we  denote the interval $(M,N)_{\mathcal M_p}$  by $(M,N)_p$; such an agreement applies to other types of intervals as well.

\begin{definition}
Let $M\in\mathcal{E}^0\cup\mathcal{E}^1$, and
  suppose that $p\in\mathbb M\cap M$. We let 
  $p^M$ denote the closure of $\mathcal M\cup\{M\}$ under intersections.
\end{definition}

The following is easy and we leave the proof to the reader.
\begin{fact}[{\cite[Lemma 1.8]{VeliVen}}]\label{Mtop-side}
Let $M\in\mathcal{E}^0\cup\mathcal{E}^1$, and
  suppose that $p\in\mathbb M\cap M$. 
\begin{enumerate}
    \item If $M\in\mathcal{E}^1$, then $p^M=\mathcal M\cup\{M\}$.
    \item If $M\in\mathcal{E}^0$, then $p^M=\mathcal M\cup\{M\}\cup\{N\cap M: N\in\mathcal{E}^1_p\}$.
    \item $p^M$ is a condition in $\mathbb M$ and extends $p$.
\end{enumerate}
\end{fact}
\begin{proof}[\nopunct]
\end{proof}

\begin{definition}
For a condition  $p\in\mathbb M$ and a model $M\in\mathcal M_p$,
  let $p\rest_M\coloneqq\mathcal M_p\cap M$.
\end{definition}
 Notice that $p\rest_M$ is in $M$, as it is a finite subset of $M$.
 If $M$ is in $\mathcal{E}^1$, then  $p\rest_M$ is  the interval $(\varnothing,M)_p$ that is an $\in$-chain, but if $M$ is countable, then it is a union of intervals.
\begin{fact}[{\cite[Fact 1.7]{VeliVen}}]\label{rest description side E0}
Suppose that $p\in\mathbb M$. Assume that $M\in\mathcal M_p$ is countable. Then
$$\mathcal M_{p}\rest_M:=\mathcal M_p\cap M=\mathcal M_p\setminus \bigcup\{[N\cap M,N)_p: N\in(\mathcal{E}^1_p\cap M)\cup \{H_\theta\}\}.$$
\end{fact}
\begin{proof}
Let $P\in\mathcal M_{p}\rest_M$. Thus $P\in M$, which in turn implies that $P$ does not belong to the interval $[M,H_\theta)_p$.
Now, let $N\in\mathcal E^1_p\cap M$. If $N\in^* P$ or $N=P$, then $P$ does not belong to the interval $[N\cap M,N)_p$.
Suppose $P\in^* N$, then $P\in N$, and hence $P\in N\cap M$, which in turn implies that $P\notin [N\cap M,N)_p$.
Therefore, the LHS is a subset of RHS. To see the other direction. Suppose $P$ does not belong to any interval as described in the above equation. In particular, $P\in^* M$.
Now, if $P\notin M$, it then means there are some models in
$\mathcal E^1_p\cap (P,M)_p$. Let $N$ be the least such model.
Then, $N\cap M\in^* P$, since otherwise by the minimality of
$N$, we  have  $P\in N\cap M\subseteq M$.
Thus $P$ belongs to $[N\cap M,N)_p$, which is a contradiction.
\end{proof}
It is not hard to see that $p\rest_M$ is an $\in$-chain.
Now, the following is immediate.
\begin{fact}\label{restriction side}
For every condition $p\in\mathbb M$ and $M\in\mathcal M_p$, $p\rest_M$ is a condition and $p\leq p\rest_M$.
\end{fact}
\begin{proof}[\nopunct]

\end{proof}

Thus we also have $\mathcal M_{{p}\upharpoonright_M}=\mathcal M_p\rest_M$! This notational equality will be useful later.

\begin{fact}[{\cite[Fact 1.12]{VeliVen}}]\label{side condition U-proper}
Suppose that $p\in\mathbb M$ and $M\in\mathcal{E}^1_p$. Then every condition $q\in M$ extending $p\rest_M$ is compatible with $p$.
\end{fact}
\begin{proof}
Let $\mathcal M_r=\mathcal M_p\cup\mathcal M_q$.
It is easy to see that $\mathcal M_r$ is closed under intersections. To see that it is an $\in$-chain, suppose that 
$P\in\mathcal M_p\setminus \mathcal M_q$ 
and $Q\in\mathcal M_q\setminus\mathcal M_p$. If $P\neq M$,
we then have  $Q\in M\in^* P$, and if $P=M$, then obviously $Q\in M$. It is clear that $r\leq p,q$.
\end{proof}

\begin{remark}\label{remark after side condition U-proper}
The above condition  is the greatest lower bound of $p$ and $q$, and denoted  by $p\land q$. Notice that
$$\mathcal M_{p\land q}=\mathcal M_p\cup\mathcal M_q$$
\end{remark} 

\begin{fact}\label{strong-properness E1}
$\mathbb M$ is strongly proper for $\mathcal{E}^1$, and hence if $\mathcal{E}^1$ is stationary, then
$\mathbb M$ preserves $\aleph_2$.
\end{fact}
\begin{proof}
Suppose that $M\in\mathcal E^1$. If $p\in M\cap\mathbb M$, then by
\cref{Mtop-side}, $p^M$ is a condition extending $p$.
Let $q\leq p^M$, then $M\in\mathcal M_q$. By \cref{restriction side}, $q\rest_M$ is a condition in $M\cap\mathbb M$. Now if $r\in M\cap\mathbb M$ extends $q\rest_M$, then $q$ is compatible with $r$ by \cref{side condition U-proper}. Thus $q$ is strongly $(M,\mathbb M)$-generic. By \cref{preserv-byS-proper} and \cref{remark-st-prop-preser}, $\mathbb P$ perseveres $\aleph_2$.
\end{proof}

\begin{lemma}[{\cite[Lemma 1.12]{VeliVen}}]\label{side condition C-proper}
Suppose that $p\in\mathbb M$. Let $M\in\mathcal{E}^0_p$. Then every condition $q\in M$ extending $p\rest_M$ is compatible with $q$. In fact, the closure of $\mathcal M_p\cup\mathcal M_q$ is a condition in $\mathbb M$, which is also the greatest lower bound of $p$ and $q$.
\end{lemma}
\begin{proof}[\nopunct]
\end{proof}
\begin{remark}\label{remark after side condition C-proper}
As before we again denote the above common extension by $p\land q$. Notice that 
$$\mathcal M_{p\land q}=\mathcal M_p\cup\mathcal M_q\cup\{N\cap M: N\in\mathcal E^1_q,~ M\in \mathcal E^0_p, \mbox{ and  N}\in M \}$$
\end{remark}

The following is similar to \cref{strong-properness E1} in light of \cref{side condition C-proper}.
\begin{fact}
$\mathbb M$ is strongly proper for $\mathcal{E}^0$. 
\end{fact}
\begin{proof}[\nopunct]

\end{proof}

\section{The Forcing Construction}\label{sec4}

In this section, we first present the phenomenon of \emph{overlapping} that was introduced by Neeman in his paper \cite{Neeman2017} regarding (partial) specialisation of trees of height and size $\omega_2$. Neeman's strategy is to attach $\mathbb S_{\omega}(T)$ to side conditions consisting of models of two types: countable and transitive, where he also requires several constraints describing the interaction of the working parts, which are elements of  $\mathbb S_{\omega}(T)$, and the models as side conditions. He then analyses this interaction. Our approach is similar to Neeman's, and we still need to require one of the fundamental constraints,  though our forcing is simpler than Neeman's. His definition of overlapping reads as follows: A model $M$ overlaps a node $t\in T\setminus M$, if there is no non-cofinal branch $b\in M$ with $t\in b$. Our terminology is different from Neeman's; we say a node  $t\in T$ is guessed in $M$ if $t$ belongs to some (non-cofinal) branch $b\in M$. 

Throughout this section, we fix a Hausdorff tree $(T,<_T)$ of height $\omega_2$ without cofinal branches.
We also fix a regular cardinal $\theta$ such that $\mathcal P(T)\in H_{\theta}$. We let $\mathcal{E}^0\coloneqq\mathcal E^0(T)$ and $\mathcal{E}^1\coloneqq\mathcal E^1(T)$ consist, respectively, of countable elementary submodels, and $\omega_1$-guessing elementary  $\rm IC$-submodels of $(H_\theta,\in,T)$. We reserve the symbols $p,q,r$ for forcing conditions, and 
$s,t,u$ for nodes in $T$.
\subsection{Overlaps Between Models and Nodes}\leavevmode

\begin{convention}
A branch through $T$ is called a $T$-branch.
\end{convention}

\begin{definition}
 Suppose that $t\in T$ and $M\in\mathcal{E}^0\cup\mathcal{E}^1$ . We abuse language and say $t$ is \emph{guessed}  in $M$ if and only if there is a $T$-branch $b\in M$ with $t\in b$.
\end{definition}
Thus every $t\in M$ is already guessed in $M$, and that no node $t$ with ${\rm ht}(t)\geq{\rm sup}(M\cap\omega_2)$ is guessed in $M$, since $M$ has no cofinal branches. We shall often use the following without mentioning.

\begin{lemma}
Suppose that $t\in T$ and $M\in\mathcal{E}^0\cup\mathcal{E}^1$. If there is $s\in M$ with
$t\leq_T s$, then $t$ is guessed in $M$.
\end{lemma}
\begin{proof}
Pick $s\in T\cap M$  with $t\leq_T s$. Then $\bar{b}_s\in M$ is a $T$-branch  and $t\in \bar{b}_s$.
\end{proof}

\begin{notation}
Assume that $t\in T$ and $M\in\mathcal{E}^0\cup\mathcal{E}^1$. Then
\begin{itemize}
    \item  $\eta_M(t)$ denotes ${\rm sup}\{{\rm ht}(s): s\in T\cap M \text{ and } s\leq_T t \}$.
\item $O_M(t)$ denotes the unique node $s\in T_{\eta_M(t)}$ such that $s\leq_T t$.
\item $b_M(t)$ denotes $b_{O_M(t)}$.
\end{itemize}

\end{notation}
Observe that $O_M(t)$ is always well-defined as $T$ is a rooted tree belonging to every model in $\mathcal{E}^0\cup\mathcal{E}^1$. By definition, we have $\eta_M(t)\leq {\rm sup}(M\cap\omega_2)$. In our analysis, we shall focus on $O_M(t)$ rather than $t$ itself. It would be useful to have this intuition that if $t\notin M$, then the node $O_M(t)$ is where $b_t$ detaches from $M$.  We shall see that if $M\in\mathcal{E}^1$, then not only $\eta_M(t)$ is less than $M\cap\omega_2$, but also if its cofinality is uncountable, then $O_M(t)$ is in $M$. 
Moreover, if $M\in\mathcal E^1$ , then $t$ is guessed in $M$ if and only if $t=O_M(t)\in M$. The situation is different for countable models,  as if $M\in\mathcal E^0$ and $t\in M$ is of uncountable height in $T$, then one can find some $s\in b_t\setminus M$. Such an $s$ is necessarily guessed in $M$ though it does not belong to $M$.

\begin{lemma}\label{eta in M}
Suppose that $t\in T$ and  $M\in\mathcal{E}^0\cup\mathcal{E}^1$.

\begin{enumerate}
    \item If $t$ is  guessed in $M$ and $\eta_M(t)\in M$, then $t\in M$
    \item If $t$ is guessed in $M$, but $\eta_M(t)\notin M$, then  ${\rm ht}(t)\leq {\rm min}(M\cap\omega_2\setminus \eta_M(t))$. 
\end{enumerate}

\end{lemma}
\begin{proof}
Of course, the first item follows from the proof of the second one, but we prefer to give  independent proofs.
\begin{enumerate}
    \item   Assume that  $b\in M$ is a $T$-branch containing $t$. Pick $s\in b\cap M$ of height $\eta_M(t)$, which is possible as $t\in b$ implies that the order-type of $b$ is at least $\eta_M(t)+1$. Thus $s\leq_T t$. On the other hand, if $s<_T t$, then there is $u\in b\cap M$ of height $\eta_M(t)+1$, but then $u\leq_T t$, which is impossible by the definition of $\eta_M(t)$. Thus $t=s\in M$.

\item We may assume that $M$ is in $\mathcal E^0$ as otherwise it is trivial. One easily observes that $\eta_M(t)$ is below ${\rm sup}(M\cap \omega_2)$ since $T$ does not have cofinal branches. Now $\eta^*\coloneqq{\rm min}(M\cap\omega_2\setminus \eta_M(t))$ is an ordinal below $\omega_2$, but above $\eta_M(t)$. Let $b\in M$ be a branch containing $t$.
Assume towards a contradiction that ${\rm ht}(t)>\eta^*$, then there is some node $s\in b$ of height $\eta^*$, and thus
$s<_T t$. It then follows that $\eta_M(t)\geq \eta^*>\eta_M(t)$, a contradiction.

\end{enumerate}
\end{proof}

The following is too easy, and  we leave the proof to the reader.
\begin{lemma}\label{successor overlap}
Suppose that $t\in T$ and  $M\in\mathcal{E}^0\cup\mathcal{E}^1$. If
$\eta_M(t)$ is a successor ordinal, then  $O_M(t)$ is in $M$.
\end{lemma}
\begin{proof}[\nopunct]
\end{proof}

In general, if the supremum in the definition of $\eta_M(t)$ is attained by an element in $T\cap M$, then that element is $O_M(t)$, which belongs to $M$. The above lemma essentially means that it does happen if $\eta_M(t)$ is a successor ordinal.
We now turn our attention to the situation where the overlaps are more complicated as $\eta_M(t)$ is a limit ordinal.

\begin{lemma}\label{successor or limit}
Suppose that $t\in T$ and  $M\in\mathcal{E}^1$. If
${\rm cof}(\eta_M(t))$ is not countable, then  $O_M(t)\in M$.
\end{lemma}
\begin{proof}
By
\cref{successor overlap}, we may assume that $\eta_M(t)$ is a limit ordinal, and thus of cofinality $\omega_1$.
Let $\eta=\eta_M(t)$. 
Since $M$ is of size $\aleph_1$ and $\omega_1\subseteq M$, we have  $b_M(t)\subseteq M$. For every countable $a\in M$, the height of nodes in 
$a\cap b_M(t)$ is bounded below $\eta$ due  to the fact that $\eta_M(t)$ has uncountable cofinality.
Thus it is easily seen that $b_M(t)$ is countably approximated in $M$. Since $M$ is an $\omega_1$-guessing model, there is $b\in M$  such that $b\cap M=b_M(t)$. 
By elementarity, $b$ is a $T$-branch, and hence it is of size $\aleph_1$ (in particular, $\eta<M\cap\omega_2$.) Thus $b\subseteq M$, which in turn implies that
 $b_M(t)=b\in M$. But then $O_M(t)\in M$ as it can be read off from
$b_M(t)$ due to the fact that $T$ is Hausdorff.

\end{proof}

\begin{corollary}\label{successor or limit-cor}
Suppose that $t\in T$ and  $M\in\mathcal{E}^1$.  Then  $\eta_M(t)$ is in $M$.
\end{corollary}
\begin{proof}
By definition $\eta_M(t)\leq M\cap\omega_2$. Since $M$ is an $\rm IC$-model with $\omega_1\subseteq M$, the ordinal $M\cap\omega_2$ is of uncountable cofinality. If $\eta_M(t)= M\cap\omega_2$, then by \cref{successor or limit}, $O_M(t)\in M$. This is a contradiction, as $M\cap\omega_2=\eta_M(t)={\rm ht}(O_M(t))\in M$! Thus $\eta_M(t)<M\cap\omega_2$, and hence $\eta_M(t)\in M$
\end{proof}

The following is  key  for us.
\begin{lemma}\label{guessing lemma}
Assume  that  $N\in\mathcal{E}^1$ and $M\in \mathcal{E}^0$ with $N\in M$.
Let $t\in T\cap N$. If $t$ is guessed in $M$, then $t$ is guessed in $N\cap M$.
\end{lemma}
\begin{proof}
Let $b\in M$ be a $T$-branch containing $t$.
Let $\gamma={\rm sup}\{{\rm ht}(s):s\in N\cap b\}$. Then $\gamma$ exists as $t\in N$ and ${\rm ht}(t)\leq\gamma$. Note that  $\gamma\in M\cap\omega_2$ by elementarity.
Observe that if $\gamma={\rm ht}(s)$, for some $s\in N\cap b$, then by elementarity, $s\in N\cap M$. We then have $t\in \overline{b}_s\in N\cap M$. Thus let us assume that the supremum $\gamma$  is not obtained by any element of $N\cap b$. In particular, ${\rm ht}(t)<\gamma$ and the cofinality of $\gamma$ is either $\omega$ or $\omega_1$.
 We consider two cases: \\

\textbf{Case 1:} ${\rm cof}(\gamma)=\omega$.\\
 By elementarity,  there is
 a strictly $<_T$-increasing sequence $\langle s_n: n\in\omega\rangle\in M$ of nodes in $b\cap N$ such that $\sup\{{\rm ht}(s_n):n\in\omega\}=\gamma$.
Since we assumed ${\rm ht}(t)<\gamma$,  there is $n$ such that $t\leq_T s_n$. Note that $s_n\in N\cap M$, and hence $t\in \overline{b}_{s_n}\in N\cap M$. Therefore, $t$ is guessed in $N\cap M$.\\

\textbf{Case 2:} ${\rm cof}(\gamma)=\omega_1$.\\
 We claim that $ b\cap T_{\leq \gamma}$ is guessed in $N$. To see this, observe that $ b\cap T_{\leq \gamma}$ is  $\omega_1$-approximated in $N$, since if $a\in N$ is a countable set, then
 there is $s\in N\cap b\cap T_{\leq\gamma}$ such that
$a\cap b\cap T_{\leq\gamma}= a\cap b_s$ (as the cofinality of $\gamma$ is $\omega_1$.) But $a\cap b_s\in N$. As $N$ is an $\omega_1$-guessing model, we have
 $ b\cap T_{\leq \gamma}$ is guessed in $N$.
 By the elementarity of $M$, there is $b^*\in N\cap M$ such that
 $b^*\cap N=b\cap T_{\leq\gamma}\cap N $. Now $t\in N\cap b\cap T_{\leq\gamma}=b^*\cap N$. Notice that, by elementarity, $b^*$ is a $T$-branch. Thus $b^*\in N\cap M$ witnesses that $t$ is guessed in $N\cap M$. 
\end{proof}

\begin{lemma}\label{delta and intersection}
Assume  that  $N\in\mathcal{E}^1$ and $M\in \mathcal{E}^0$ with $N\in M$.
Let $t\in T\cap N$. Then $\eta_{N\cap M}(t)=\eta_M(t)$, and hence $O_{N\cap M}(t)=O_M(t)$.
\end{lemma}
\begin{proof}
Since $N\cap M\subseteq M$, $\eta_{N\cap M}(t)\leq \eta_M(t)$. Assume towards a contradiction that the equality fails. Thus, there is some $s\in M$ whose height is above $\eta_{N\cap M}(t)$ such that
$s\leq_T O_M(t)\leq_T t$. Then $s\in N$ as $\omega_1\cup\{t\}\subseteq N$. Therefore, $s\in N\cap M$, and hence ${\rm ht}(s)\leq \eta_{N\cap M}(t)$, a contradiction.
Since both $O_{N\cap M}(t)$ and $O_M(t)$ are below $t$ and of the same height, they are equal.
\end{proof}

\subsection{The Forcing Construction and its Basic Properties}\label{sub2}\leavevmode

We are now ready to define our  forcing notion $\mathbb P_T$ to specialise $T$ in generic extensions.

\begin{definition}[$\mathbb P_T$]\label{main forcing 2}
A condition in $\mathbb P_T$ is a pair $p=(\mathcal M_p,f_p)$ satisfying the following items.
\begin{enumerate}
    \item\label{1} $\mathcal M_p\in\mathbb M\coloneqq\mathbb M(\mathcal E^0,\mathcal E^1)$.
    \item\label{2} $f_p\in \mathbb S_{\omega}(T)$.
     \item \label{closed under f} For every $M\in\mathcal{E}^0_p$, if $t\in \dom(f_p)\cap M$, then $f_p(t)\in M$.
     \item\label{4} For every  $M\in\mathcal{E}^0_p$  and every  $t\in{\rm dom}(f_p)$ with $f_p(t)\in M$, if $t$ is guessed in $M$, then  $t\in M$.
   \

\end{enumerate}

We say $p$ is stronger than $q$  if and only if the following are satisfied.
\begin{enumerate}
    \item $\mathcal M_p\supseteq \mathcal M_q$.
    \item $f_p\supseteq f_q$.
\end{enumerate}

\end{definition}

Given a condition $p$ in $\mathbb P_T$ and a model $M\in\mathcal{E}^0\cup\mathcal{E}^1$ containing $p$, we define an extension of $p$ that will turn later to be generic for the relevant models.

\begin{definition}
Suppose that $M\in\mathcal{E}^0\cup\mathcal{E}^1$ and
$p\in M\cap\mathbb P_T$. We let $p^M$ be defined by $(\mathcal M^M_p,f_p)$.
\end{definition}
Recall that $\mathcal M_p^M$ is the closure of $\mathcal M_p\cup\{M\}$ under intersections (see \cref{Mtop-side}.)
\begin{proposition}\label{Mtop-P}
Suppose that $M\in\mathcal{E}^0\cup\mathcal{E}^1$ and
$p\in M\cap\mathbb P_T$.
Then $p^M$ is a condition extending $p$ such that $M\in\mathcal M_{p^M}$.
\end{proposition}
\begin{proof}
We check \cref{main forcing 2} item by item.
\cref{1}  is essentially \cref{Mtop-side}. \cref{2} is obvious of course.
 To see \cref{closed under f,4} hold true, let $N\in\mathcal{E}^0_{p^M}$. We may assume that $N\notin\mathcal M_p$.
 Therefore, the only interesting case is  $M\in\mathcal E^0$  and  $N=P\cap M$, for some $P\in\mathcal{E}^1_p$. Thus fix such models.\\
 
\textbf{\cref{closed under f}:} Let $t\in{\rm dom}(f_{p^M})\cap N$. We have $f_p(t)\in M$, as  $p\in M$, and also we have $f_p(t)\in P$, as $\omega_1\subseteq P$. Thus $f_p(t)\in P\cap M=N$.\\

\textbf{\cref{4}:}  Let $t\in{\rm dom}(f_p)$ be such that $f_p(t)\in N$.
 If there is a $T$-branch $b\in N$ with $t\in b$, then $t\in P$ (since $b\subseteq P$), and hence
$t\in P\cap M=N$.\\

Finally, by the construction of $p^M$, we have $M\in\mathcal M_{p^M}$, and by  \cref{Mtop-side},   $p^M\leq p$.
\end{proof}

We now define the restriction of a condition to a  model in the side conditions coordinate.

\begin{definition}[restriction]
Suppose that $p\in\mathbb P_T$ and $M\in\mathcal M_p$.
We let the \emph{restriction} of $p$ to $M$ be $p\rest_M=(\mathcal{M}_{p\upharpoonright_M},f_p\rest_M)$, where $f_p\rest_M$ is the restriction of the function $f_p$ to $\dom(f_p)\cap M$.
\end{definition}

Recall that $\mathcal{M}_{p\upharpoonright_M}=\mathcal M_p\cap M$.
Observe that if $M$ is in $\mathcal E^0$, then by \cref{closed under f} of \cref{main forcing 2}, $f_{p\upharpoonright_M}=f_p\cap M$.  This  is trivial for  models in $\mathcal E^1$.

\begin{proposition}\label{rest to M in P}
Suppose that $p\in\mathbb P_T$ and  $M\in\mathcal M_p$. Then $p\rest_M\in \mathbb P_T\cap M$ and $p\leq p\rest_M$.
\end{proposition}
\begin{proof}
We check \cref{main forcing 2} item by item.
By \cref{restriction side}, $\mathcal M_{p\upharpoonright_M}$ is an $\in$-chain and closed under intersections, and hence it is in $\mathbb M$. By \cref{closed under f} of \cref{main forcing 2},  $f_p\cap M$ is in  $\mathbb S_{\omega}(T)$. Observe that $M$ contains  $p\rest_M$,  as it is a finite subset of  $M$.
\cref{closed under f,4} remain valid since all models in $\mathcal M_{p\upharpoonright_M}$ and all nodes in ${\rm dom}(f_{p\upharpoonright_M})$ are, respectively, in $\mathcal M_p$
 and ${\rm dom}(f_p)$.
It is easy to see  that $p$ extends $p\rest_M$.
\end{proof}

\begin{notation}
For a condition $p\in \mathbb P_T$, a model $M\in\mathcal M_p$, and a condition  $q\in M\cap \mathbb P_T$ with
 $q\leq p\rest_M$, we let $p\land q$ denote the pair $(\mathcal M_p\land\mathcal M_q, f_p\cup f_q)$.
\end{notation}
Note that  $p\land q$ is not necessarily  a condition, however we shall use it as a pair of objects.
Notice that $\mathcal M_{p\land q}$ is the closure of $\mathcal M_p\cup\mathcal M_q$ under intersections, and belongs to $\mathbb M$ (see \cref{remark after side condition U-proper} and \cref{remark after side condition C-proper},) and that also
$f_{p\land q}$ is a well-defined function due to the fact that $p$ satisfies \cref{closed under f} of \cref{main forcing 2}.

\begin{lemma}\label{closed under f lemma}
Suppose $p$ is a condition in $\mathbb P_T$ and $M$ is a model in $\mathcal M_p$.
Assume that $q\in M\cap\mathbb P_T$ extends $p\rest_M$. Then $p\land q$ satisfies \cref{closed under f} of \cref{main forcing 2}.
\end{lemma}
\begin{proof}
Fix $N\in\mathcal E^0_{p\land q}$ and
$t\in{\rm dom}(f_p)\cup{\rm dom}(f_q)$. Assume that $t$ is in $N$. We shall show that $f_{p\land q}(t)\in N$. We split the proof into two cases.\\ 

\textbf{Case 1:} $M$ is in $\mathcal{E}^1$.\newline
  In this case, $\mathcal M_{p\land q}=\mathcal M_p\cup\mathcal M_q$, by \cref{remark after side condition U-proper}.  If $N\in\mathcal M_q$, then $t\in N\subseteq M$, and hence $t\in{\rm dom}(f_q)$. Thus $f_{p\land q}(t)=f_q(t)\in N$.
  Now suppose that $N\in\mathcal M_p\setminus \mathcal M_q$. We may assume  $t\in{\rm dom}(f_q)$. Therefore, in $\mathcal M_p$, we have $M\in^* N$, which in turn implies that there is $M'\in\mathcal{E}^1_p$ such that $M\subseteq M'\in N$ and $M'\cap N\in M$.
Then, $M'\cap N\in\mathcal M_q$ and $t\in M'\cap N$. As $q$ is a condition, we have $f_{p\land q}(t)=f_q(t)\in M'\cap N\subseteq N$.\\

\textbf{Case 2:} $M$ is in $\mathcal{E}^0$.\newline
    Observe that it  is enough to assume $N\in\mathcal M_p\cup\mathcal M_q$: if $N\in\mathcal M_p\land \mathcal M_q$, then $N=P\cap N'$, for some  $P'\in\mathcal M_p\cup\mathcal M_q$, and some  $N'\in\mathcal M_p\cup\mathcal M_q$. By our assumption, $f_{p\land q}(t)$ belongs to $N'$, and hence, $f_{p\land q}(t)\in P'\cap N'=N$, as $\omega_1\subseteq P'$.
     
    As in the previous case, we may assume $t\in{\rm dom}(f_q)$ and $N\in\mathcal M_p\setminus \mathcal M_q$. Let us first assume that $N\in^* M$. Suppose that $N$ is the minimal counter-example with the above properties. Thus  there is $P\in \mathcal{E}^1_p\cap M$ such that
    $N\in [P\cap M,P)_p$.  Now  $P\cap M\nsubseteq N$, as otherwise $f_q(t)\in N$, since $t\in P\in\mathcal M_q$ and $f_q(t)\in P\cap M$. Therefore,  there is some $Q\in N$ such that $Q\cap N\in^* P\cap M\in Q$. 
    Notice that $t\in P$, and hence $t\in P\cap M\subseteq Q$. Thus $t\in Q\cap N$.
    Now $Q\cap N$ is also a counter-example to our claim, since $t\in Q\cap N\subseteq N$, $Q\cap N\in\mathcal M_p\setminus\mathcal M_q$ (as otherwise, we would have $f_q(t)\in Q\cap N\subseteq N$), and $Q\cap N\in^* M$. This contradicts  our minimality assumption.
    
Two cases remain.   The case $N=M$ is trivial, and thus we only need to assume  that $M\in^* N$. If $M\subseteq N$, then $f_q(t)\in N$. And if $M\nsubseteq N$, then  there is some $P\in\mathcal{E}^1_p$ such that $P\cap N\in^* M\in P\in N$ (see \cref{smallremark}.) Notice that $t\in P\cap N$. Thus by the previous paragraph,
   $f_q(t)\in P\cap N\subseteq N$.

\end{proof}

\subsection{Preserving \texorpdfstring{$\aleph_2$}{}}\leavevmode

In this subsection, we prove that $\mathbb P_T$ preserves the regularity of $\aleph_2$. With a similar idea, we shall establish the properness of $\mathbb P_T$ in the subsequent subsection.

\begin{lemma}\label{compatibility lemma 1-U}
Suppose $p$ is a condition in $\mathbb P_T$ and that $M\in\mathcal{E}^1_p$.
Assume that $q\in M$ is a condition extending $p\rest_M$. Then $p\land q$ satisfies \cref{4} of \cref{main forcing 2}.
\end{lemma}
\begin{proof}
Set $r=p\land q$.
Notice that $f_r$ is well-defined as a function. Now fix $t\in{\rm dom}(f_r)$ and $N\in\mathcal E^0\cap \mathcal M_r$ so that $f_r(t)\in N$. We shall show that if $t$ is guessed in $N$, then $t\in N$. Notice that by \cref{remark after side condition U-proper},  we have $\mathcal M_r=\mathcal M_p\cup\mathcal M_q$.
We shall consider the nontrivial cases:\\

\textbf{Case 1:} $t\in{\rm dom}(f_p)$ and $N\in\mathcal M_q\setminus \mathcal M_p$.

Assume that $t$ is guessed in $N$. Thus there is a $T$-branch $b\in N\subseteq M$ with $t\in b$. As $b$ is of size $\leq\!\aleph_1$ and $\omega_1\subseteq M$, we have  $t\in b\subseteq M$. Thus $t\in M$, which in turn implies that $t\in {\rm dom}(f_q)$ and $f_q(t)=f_p(t)=f_r(t)\in N$. But then $t\in N$, as $q$ is a condition.\\

\textbf{Case 2:}
$t\in{\rm dom}(f_q)\setminus {\rm dom}(f_p)$ and $N\in\mathcal M_p\setminus \mathcal M_q$.

   In this situation, $N$ is not in $M$ since $\mathcal M_q\supseteq \mathcal M_p\cap M$, and hence there is some  $M'\in\mathcal E^1_p$ with $M\subseteq M'\in N$ such that $M'\cap N\in M$. Note that $t\in M'$. Assume that $t$ is guessed in $N$. By \cref{guessing lemma}, $t$ is guessed in $M'\cap N$. On the one hand, $f_q(t)=f_r(t) $ belongs to $ M'\cap N$, and that $M'\cap N\in M\cap\mathcal M_p\subseteq\mathcal M_q$. Since $q$ is a condition, we have $t\in M'\cap N\subseteq N$.

Thus far, we have shown that $p\land q$ satisfies all items in \cref{main forcing 2}, possibly except \cref{2}.
We shall show that there are situations $p\land q$ is indeed a condition. We now prepare the ground for this.
\end{proof}

\begin{definition}
 For a conditions $p\in\mathbb P_T$ and a model $M\in\mathcal{E}^1_p$,
  we let 
  $$\mathscr{D}(p,M)=\{t\in{\rm dom}(f_p): t\notin M\}.$$
\end{definition}

\begin{definition}[$M$-support]\label{M-support U def}
  Suppose $p$ is a condition in $\mathbb P_T$ and that $M\in\mathcal{E}^1_p$.
  We say that a function $\sigma:\mathscr{D}(p,M)\rightarrow T\cap M$  is an \emph{$M$-support} for $p$ if the following hold, for every $t\in\dom(\sigma)$.

 \begin{enumerate}
     \item If $O_M(t)\in M$, then  $\sigma(t)=O_M(t)$.
     \item  If $O_M(t)\notin M$, then   $\sigma(t)<_T O_M(t)$ is such that there is no node in ${\rm dom}(f_p)$ whose height belongs to the interval $\big[{\rm ht}(\sigma(t)),\eta_M(t)\big)$.
 \end{enumerate}

\end{definition}

\begin{lemma}\label{existnece M-suppor U}
Suppose $p$ is a condition in $\mathbb P_T$.
Assume that $M\in\mathcal{E}^1_p$. Then, there is an $M$-support  $\sigma $ for $p$.
\end{lemma}
\begin{proof}
Fix $p\in\mathbb P_T$.
It is enough to define $\sigma$ for  $t\in\mathscr D(p,M) $ with $O_M(t)\notin M$. Thus fix such a $t$.
Notice that   ${\rm dom}(f_p)$ is finite, and that, by \cref{successor overlap},  $\eta_M(t)$ is a limit ordinal. Thus
one may easily find a node $\sigma(t)$ with the above properties.
\end{proof}

\begin{definition}[$M$-reflection]\label{def M-reflection U}
  Suppose that $p\in\mathbb P_T$ and  $M\in\mathcal{E}^1_p$.
  A condition $q$ is called an \emph{$(M,\sigma)$-reflection} of $p$, where $\sigma$ is  an $M$-support   for $p$, if the following properties are satisfied.
  \begin{enumerate}
  \item\label{p1-U-reflection} $q\leq p\rest_M$.
  \item For every $t\in \dom(\sigma)$, the following hold:
  \begin{enumerate}
      
  \item\label{p2-U-reflection} There is no node in ${\rm dom}(f_q)$ whose height is the interval $\big[{\rm ht}(\sigma(t)),\eta_M(t)\big)$.
     \item\label{p3-U-reflection} For every $s\in{\rm dom}(f_q)$, if $s<_T \sigma(t)$,  then
 $f_q(s)\neq f_p(t)$.
 \end{enumerate}
  \end{enumerate}
  
  Let $R_p(M,\sigma)$ be the set of $(M,\sigma)$-reflections of $p$ with support $\sigma$.
\end{definition}

\begin{remark}
Notice that if $M^*\prec H_{\theta^*}$, for some sufficiently large regular cardinal $\theta^*$, which contains $T$ and $H_{\theta}$,
and that $p$ is a condition in $\mathbb P_T$ with $M\coloneqq M^*\cap H_{\theta}\in\mathcal{E}^1_p $, then $R_p(M,\sigma)\in M^*$, whenever $\sigma$ is an $M$-support for $p$. 
\end{remark}

\begin{lemma}\label{Mreflection-Ulemma}
Let $p\in\mathbb P_T$. Assume that  $M\in\mathcal{E}^1_p$, and let $\sigma$ be an $M$-support  for $p$. Then   $p\in R_p(M,\sigma)$.
\end{lemma}
\begin{proof}
We check the items in \cref{def M-reflection U}.
\cref{p1-U-reflection} is essentially \cref{rest to M in P}.
\cref{p2-U-reflection} follows from the definition of $\sigma$. \cref{p3-U-reflection} follows from the fact that $p$ is a condition, and that $\sigma(t)<_T t$. 
\end{proof} 

\begin{lemma}\label{claim1- gen U}
Suppose  $p$ is a condition in $\mathbb P_T$. Let  $M\in\mathcal{E}^1_p$, and let
$q\in M$ be an $(M,\sigma)$-reflection of $p$, for some $M$-support $\sigma$ for $p$. Let $r=p\land q$. Then  $f_r\in B_{\omega}(T)$.
\end{lemma}
\begin{proof}
Since $q\leq p\rest_M$, $f_r$ is well-defined as a function.
We shall show that it satisfies the specialising property. To do this, we only discuss the  nontrivial case by considering
two arbitrary comparable nodes $t\in{\rm dom}(f_p)\setminus{\rm dom}(f_q)$ and $s\in{\rm dom}(f_q)\setminus {\rm dom}(f_p)$. 
We claim that $f_r(t)\neq f_r(s)$. Observe that $s\in M$. The fact that $M\cap\omega_2$ is an ordinal imply that if $t\leq _T s$, then $t\in M$, which is a contradiction as $t\notin\dom(f_q)$. Thus, the only possibility is $s<_T t$.  Since $q\in R_p(M,\sigma)\cap M$, the height of $s$ is not in the interval $\big[{\rm ht}(\sigma(t)),\eta_M(t)\big)$. Thus  $s<_T\sigma(t)$. Then \cref{p3-U-reflection} of \cref{def M-reflection U} implies that
  $f_q(s)\neq f_p(t)$. Therefore, $f_r(t)\neq f_r(s)$.
\end{proof}
We have now all the necessary tools to prove the preservation of $\aleph_2$ by $\mathbb P_T$.
\begin{lemma}\label{genericity-U}
Suppose $p$ is a condition in $\mathbb P_T$. Assume that $\theta^*$ is a sufficiently large regular cardinal, and that $M^*\prec H_{\theta^*}$ contains the relevant objects. Suppose that $M\coloneqq M^*\cap H_{\theta}$ is in $\mathcal{E}^1_p$.
Then, $p$ is $(M^*,\mathbb P_T)$-generic.
\end{lemma}
\begin{proof}
Fix $p'\leq p$. Then $M\in\mathcal M_{p'}$. Thus we may assume that $p=p'$.
Let $D\in M^*$ be a dense subset of $\mathbb P_T$.
We may also assume that $p\in D$. By  \cref{existnece M-suppor U,Mreflection-Ulemma}, there exists an $M$-support  of  $p$, say $\sigma$, such that $p\in R_p(M,\sigma)$. 
Notice that $R_p(M,\sigma)$ is in $M^*$. Thus by elementarity, there is some $q\in D\cap R_p(M,\sigma)\cap M$.
Set $r=p\land q$. 
 Now, \cref{side condition U-proper,claim1- gen U,closed under f lemma,compatibility lemma 1-U} imply  that $r$ satisfies \cref{1,2,closed under f,4} of  \cref{main forcing 2}, respectively.
It is clear that $p\land q$ extends both $p$ and $q$.

\end{proof}

\begin{corollary}\label{aleph2 preserv}
Assume ${\rm GM}^*(\omega_2)$. Then $\mathbb P_T$ preserves $\aleph_2$.
\end{corollary}
\begin{proof}
Let $\theta^*$ be a sufficiently large regular cardinal. By \cref{preserv-byS-proper},
it is enough to show that for stationary many models $M$ in $H_{\theta^*}$, of size $\aleph_1$, every condition in $M$ can be extended to an $(M,\mathbb P_T)$-generic condition. Let 
\[
\mathcal S=\{M\prec H_{\theta^*}: \mathcal E^1,\mathcal E^0,T,\theta\in M \text{ and }M\cap H_\theta\in\mathcal{E}^1\}.
\]
By ${\rm GM}^*(\omega_2)$,
$\mathcal S$ is stationary in $\mathcal P_{\omega_2}(H_{\theta^*})$.
Now let $M^*\in\mathcal S$ and $p\in\mathbb P_T\cap M^*$. Set $M=M^*\cap H_\theta$. By \cref{Mtop-P}, 
$p^M$ is a condition with $p^M\leq p$, and by \cref{genericity-U} it is $(M^*,\mathbb P_T)$-generic.
\end{proof}

\subsection{Properness}\leavevmode

This subsection is devoted to the proof of the properness of $\mathbb P_T$. We will closely follow our strategy in the previous subsection. Notice that our notation and definition related to models in $\mathcal E^0$ are similar to the ones we used for the preservation of $\aleph_2$, but hopefully  there will be no confusion, since these two parts are completely independent,

\begin{lemma}\label{compatibility lemma 1-C}
Suppose $p$ is a condition in $\mathbb P_T$ and that $M\in\mathcal{E}^0_p$.
Assume that $q\in M$ is a condition extending $p\rest_M$. Then $p\land q$ satisfies \cref{4} of \cref{main forcing 2}.
\end{lemma}
\begin{proof}
Set $r=p\land q$.
Notice that $f_r$ is well-defined as a function. Fix $t\in{\rm dom}(f_r)$ and $N\in\mathcal E^0\cap\mathcal M_r$ so that $t$ is guessed in $N$ and $f_r(t)\in N$. We shall show that $t\in N$.  As in \cref{compatibility lemma 1-U}, we shall study the nontrivial cases, thus we may  assume that either  $t\in{\rm dom}(f_q)$ and $N\notin \mathcal M_q$, or $t\in{\rm dom}(f_p)$ and $N\notin \mathcal M_p$. 
Since $M$ is in $\mathcal{E}^0$,  the proof  consists of three cases as $\mathcal M_r\setminus (\mathcal M_p\cup\mathcal M_q)$ may be nonempty.
Recall that by \cref{remark after side condition C-proper}, $\mathcal M_r$ is the union of
$\mathcal M_p\cup\mathcal M_q$ and the set of models of the form $P\cap Q$, where $P\in Q$ are in $\mathcal E^1_q$  and $\mathcal E^0_p$, respectively.\\

\textbf{Case 1:} $t\in{\rm dom}(f_q)$ and $N\in\mathcal M_p\setminus \mathcal M_q$.\newline
    In this situation, we have  $N\in (P\cap M,P]_p$ for some $P\in(\mathcal{E}^1_p\cap  M)\cup\{H_\theta\}$. Since $t$ is guessed in $N\subseteq P$ and $\omega_1\subseteq P$, we  have  $t\in P$. 
   Assume towards a contraction that $t\notin N$. We may assume that $N$ is the least model in $\mathcal M_p$ with the above properties.
   This implies that $P\cap M\nsubseteq N$, since $t\in P\cap M$. Therefore, by \cref{smallremark}, 
    there is a model $Q\in \mathcal E^1_p$ such that
    $P\cap M\in Q\in N\in P$ and $Q\cap N\in^* P\cap M$.
   Observe that $t\in Q$. By \cref{guessing lemma}, $t$ is guessed in $Q\cap N$.  On the other hand $f_q(t)\in Q\cap N$.
   Since $t\notin Q\cap N$,
     our minimality assumption implies that $Q\cap N$ is in $\mathcal M_q$, but then  since $q$ is a condition, $t$ is an element of $Q\cap N\subseteq N$, a contradiction!\\

\textbf{Case 2:} $t\in{\rm dom}(f_p)$ and $N\in\mathcal M_q$.\newline
  We have  $f_p(t)\in N\subseteq M$.   Observe that $t$ is also guessed in $M$, since $N\subseteq M$. As $p$ is a condition,  \cref{4} of \cref{main forcing 2} implies that $t\in M\cap {\rm dom}(f_p)\subseteq {\rm dom}(f_q)$. On the other hand, $q$ is a condition and $N\in\mathcal M_q$, and hence, by \cref{4} of \cref{main forcing 2}, $t\in N$.\\
    
\textbf{Case 3:} $t\in{\rm dom}(f_r)$ and $N\in \mathcal M_r\setminus(\mathcal M_p\cup \mathcal M_q)$.\newline
    There are  $P\in \mathcal E^1_q$ and  $Q\in\mathcal E^0_p$  with $P\in Q$ such that $N=P\cap Q$. Let $b\in N$ be a $T$-branch with $t\in b$. Then $t$ is guessed in $Q$, as $b\in Q$. We have also $f_p(t)\in Q$. Thus by the two previous cases, $t\in Q$. On the other hand, $b\in P$ and  $b\subseteq P$, as $T$ has no cofinal branches, and $P\cap\omega_2$ is an ordinal. Thus
     $t\in P$. Therefore, $t\in P\cap Q=N$.

\end{proof}

\begin{notation}
Assume that $p$ is a condition in $\mathbb P_T$, and that $M\in\mathcal E^0_p$.
\begin{enumerate}
    \item  We let $\mathscr{D}(p,M)$ denote the set of $t\in{\rm dom}(f_p)$ such that $t\notin M$, but $f_p(t)\in M$.
    \item $\mathscr{O}(p,M)\coloneqq\{t\in \mathscr{D}(p,M):
   O_M(t) \text{ is not guessed in } M \text{ and } \eta_M(t)\notin M \}$.
\end{enumerate}
\end{notation}

\begin{definition}[$M$-support]\label{def-C-sigma}
  Suppose $p$ is a condition in $\mathbb P_T$ and $M\in\mathcal{E}^0_p$.
  We say a  function $\sigma:\mathscr{D}(p,M)\rightarrow M$ 
  is an \emph{$M$-support} for $p$ if the following hold, for every $t\in\dom(\sigma)$.
 \begin{enumerate}
     \item If $O_M(t)$ is guessed in $M$, then $\sigma(t)\in M$  is such that
     $M\cap \sigma(t)= M\cap b_M(t)$. 
     \item\label{2-def-C-sigma}  If $O_M(t)$ is not guessed in $M$, then $\sigma(t)\subseteq b_M(t)$ is a $T$-branch in $M$ such that no node in ${\rm dom}(f_p)$ has height in the interval $\big[{\rm ht}(\sup(\sigma(t))),\eta_M(t)\big)$.
 \end{enumerate}

\end{definition}

Note that if $t\in \dom(\sigma)$ and $O_M(t)$ is guessed in $M$, then by elementarity, $\sigma(t)$  is a $T$-branch, in fact it is a cofinal branch through  $ T_{<\eta^*_M(t)}$, where $\eta^*_M(t)={\rm min}(M\cap\omega_2\setminus\eta_M(t))$. Moreover, $\sigma(t)$ is unique.

\begin{lemma}\label{existence M-support}
Let $p\in\mathbb P_T$,  and let $M\in\mathcal E^0_p$. Then, there is an $M$-support for $p$.
\end{lemma}
\begin{proof}
Suppose that $t\in\mathscr{D}(p,M)$. If $O_M(t)$ is guessed in $M$, then there is a $T$-branch $b\in M$ such that $O_M(t)\in b$. Let $\eta^*_M(t)={\rm min}(M\cap\omega_2\setminus\eta_M(t))$, and
 set $\sigma(t)\coloneqq b\cap T_{<\eta^*_M(t)}$. It is easily seen that
$M\cap \sigma(t)=M\cap b_M(t).$

If $O_M(t)$ is not guessed in $M$, then $\eta_M(t)$ is a  limit ordinal by \cref{successor overlap}. Since $\dom(f_p)$ is finite,  there is a  sequence of nodes in $M$ cofinal in $O_M(t)$. Thus one can find an ordinal $\gamma\in M$, such that
there is no node in ${\rm dom}(f_p)$ whose height is in the interval $[\gamma,\eta_M(t))$.
Choose a node $s$ of height $\gamma$ below $O_M(t)$ and set $\sigma(t)\coloneqq\overline{b}_{s}$. We have $s\in M$, since $\gamma\in M$. Thus $\sigma(t)\in M$. Observe that 
${\rm ht}(\sup(\sigma(t)))={\rm ht}(s)=\gamma$.
\end{proof}
\begin{definition}[$M$-reflection]\label{M-ref C def}
  Suppose $p$ is a condition in $\mathbb P_T$. Assume that $M\in\mathcal{E}^0_p$.
  Let $\sigma$ be an $M$-support for $p$.
  A condition $q$ is called an \emph{$(M,\sigma)$-reflection of $p$}  if the following properties are satisfied. 
  \begin{enumerate}
      \item\label{p1-C-Reflection} $q\leq p\rest_M$.
      \item \label{p22-C-Reflection}The following hold for every $t\in\dom(\sigma)$.
      \begin{enumerate}
        
      \item\label{p2-C-Reflection} If $\eta_M(t)\in M$, then there is no node in ${\rm dom}(f_q)$ whose height belongs to
      the interval $\big[{\rm ht}(\sup(\sigma(t))),\eta_M(t)\big)$.
      \item\label{p3-C-Reflection} For every $s\in{\rm dom}(f_q)$ with $s\in \sigma(t)$,   $f_q(s)\neq f_p(t)$. 
  \end{enumerate}
    
      \end{enumerate}
  
  Let $R_p(M,\sigma)$ denote the set of $(M,\sigma)$-reflections of $p$.
\end{definition}

Notice that as before, if $M^*\prec H_{\theta^*}$, for some sufficiently large regular cardinal $\theta^*$ which contains $T$ and $H_\theta$,
and $p$ is a condition in $\mathbb P_T$ with $M\coloneqq M^*\cap H_{\theta}\in\mathcal{E}^0_p $, then $R_p(M,\sigma)\in M^*$, whenever $\sigma$ is an $M$-support for $p$. 

\begin{lemma}\label{reflection C existence}
Suppose $p$ is a condition in $\mathbb P_T$, and that $M\in\mathcal{E}^0_p$.
 Let $\sigma$ be an $M$-support set for $p$. Then   $p\in R_p(M,\sigma)$.
\end{lemma}
\begin{proof}
Let us check the items in \cref{M-ref C def}.
\cref{p1-C-Reflection} is essentially \cref{rest to M in P}. To verify \cref{p22-C-Reflection}, let us fix $t\in\dom(\sigma)$.\\

\textbf{\cref{p2-C-Reflection}:}
Assume that $\eta_M(t)\in M$. If $O_M(t)$ is not guessed in $M$, then  by 
the \cref{2-def-C-sigma} of \cref{def-C-sigma}, there is no node in ${\rm dom}(f_p)$ with height in the interval $\big[{\rm ht}(\sup(\sigma(t))\big),\eta_M(t))$. Thus let us assume that  $O_M(t)$ is guessed in $M$. We show that $\sigma(t)=b_M(t)$, which in turn implies that the interval
$\big[{\rm ht}(\sup(\sigma(t))),\eta_M(t)\big)$ is empty.  To show that $\sigma(t)=b_M(t)$, it is enough to show that $b_M(t)\in M$. 
Suppose $b\in M$ is a $T$-branch with $O_M(t)\in b$. Then the order type of $b$ is at least $\eta_M(t)+1$ and $O_M(t)$ is the $\eta_M(t)$-th element of $b$. Since  $\eta_M(t)\in M$, we have $O_M(t)\in M$, and hence $b_M(t)\in M$.\\

\textbf{\cref{p3-C-Reflection}:}
Suppose  that $s\in \sigma(t)$ and $f_p(s)=f_p(t)$. Then $s$ is guessed in $M$. As  $f_p(t)\in M$ and $p$ is a condition, we have  $s\in M$.
This implies that $s\leq_T O_M(t)\leq_T t$. Since $p$ is a condition, we  $t=s\in M$, which is a contradiction! (as $t\notin M$.)
 
\end{proof}

\begin{lemma}\label{compatibility lemma 2-C}
Suppose $p\in\mathbb P_T$, and that $M\in\mathcal{E}^0_p$.
Assume that $q\in M\cap R_p(M,\sigma)$.
Let $r\coloneqq p\land q$. Then  $r'=(\mathcal M_r, f_r\setminus \{(t,f_p(t)): t\notin \mathscr{O}(p,M)\})$ is a condition.
\end{lemma}
\begin{proof}
\cref{side condition C-proper,closed under f lemma,compatibility lemma 1-C}
imply that
$r'$ satisfies \cref{1,closed under f,4} of \cref{main forcing 2}, respectively.
Therefore, it remains to show that the well-defined function $f_{r'}\coloneqq f_r\setminus \{(t,f_p(t)): t\notin \mathscr{O}(p,M)\}$ is  a condition in $\mathbb S_{\omega}(T)$. To see this, let $s\in{\rm dom}(f_q)\setminus {\rm dom}(f_p)$ and $t\in{\rm dom}(f_{r'})\setminus {\rm dom}(f_q)$.
Assume that $s$ and $t$ are comparable in $T$, we shall show that
$f_q(s)\neq f_p(t)$. We may assume that $f_p(t)\in M$.
Thus $t<_T s$ is impossible, as otherwise 
$t$ is guessed in $M$, and hence $t\in M$, which is a contradiction!
 Consequently, the only possible case is $s<_T t$.
 In this case, $s<_T O_M(t)$. We claim that $s\in \sigma(t)$.
 This is clear if  $O_M(t)$ is guessed in $M$.
  If $O_M(t)$ is not guessed in $M$, then $\eta_M(t)\in M$ as $t\notin\mathscr O(p,M)$. Therefore, by \cref{p2-C-Reflection} of \cref{M-ref C def}, the height of $s$ avoids the interval $\big[{\rm ht}(\sup(\sigma(t))),\eta_M(t)\big)$.
 Thus $s<_T \sup(\sigma(t))$, and hence $s\in\sigma(t)$. In either case, $s\in \sigma(t)$, but then \cref{p3-C-Reflection}  of \cref{M-ref C def} implies that $f_p(t)\neq f_q(s)$.

\end{proof}

\begin{proposition}\label{genericity-C}
Suppose that $p\in\mathbb P_T$. Let $\theta^*$ be a sufficiently large regular cardinal. Assume that $M^*\prec H_{\theta^*}$ is countable and contains $T$ and $\theta$.
If $M\coloneqq M^*\cap H_{\theta}\in\mathcal M_p$. Then $p$ is $(M^*,\mathbb P_T)$-generic.
\end{proposition}
\begin{proof}
Assume that $p'\leq p$. Since $M\in\mathcal M_{p'}$, we may assume without loss of generality that $p'=p$.
Let $D\in M^*$ be a dense subset of $\mathbb P_T$.
We may also assume, without loss of generality, that $p\in D$.
Since $M^*$ is fixed throughout proof, we simply denote $\eta_M(t)$  by $\eta_t$.
By  \cref{existence M-support,reflection C existence}, there is an $M$-support $\sigma$ for $p$ so that 
$p\in R_p(M,\sigma)$. Observe  that $R_p(M,\sigma)\in M^*$.
Let  $\langle t_i: i< m\rangle $ enumerate $\mathscr{O}(p,M)$ so that $\eta_{t_i}\leq\eta_{t_{i+1}}$, for every $i<m-1$. Let $\langle \eta_i: i< m'\rangle$ be the strictly increasing enumeration of
$\{ \eta_{t_i}:i< m\}$. To reduce the amount of notation, we may assume that $m=m'$.   For every $i< m$, set
\[
\eta^*_i={\rm min}(M\cap(\omega_2+1)\setminus\eta_i).
\]  
Notice that $\eta_i^*<\eta_{i+1}$, for every $i<m-1$.
For every $i<m$, we let  also $\hat{t_i}$ denote $\sup(\sigma(t_i))$.
Note that $\hat{t_i}$ exists, as $t_i\in \mathscr{O}(p,M)$.
Let us call a map $ x\mapsto p_x$  from  $\mathcal P_{\omega_1}(T)$ into $ \mathbb P_T$, a $T$-assignment if  the following properties are satisfied for every $x\in \mathcal P_{\omega_1}(T)$.
\begin{enumerate}
\item $p_x\in R_p(M,\sigma)\cap D$.
\item  $ |\dom(f_{p_x})|=|\dom(f_p)|$.
\item\label{3-proper} For every $s\in{\rm dom}(f_{p_x})$ and every $i< m$, if ${\rm ht}(s)\in\big[{\rm ht}(\hat{t}_i),\eta^*_i\big)$, then 
    $${\rm sup}\{{\rm ht}(u): u\in x\cap T_{< \eta^*_{i}}\}<{\rm ht}(s).$$

\end{enumerate}

We first show that there are $T$-assignments in $M^*$.
\begin{claim}
There is a $T$-assignment in $M^*$.
\end{claim}
\begin{proof}
We observe that all the parameters in the above properties are in $M^*$.
By elementarity and the Axiom of Choice, it is enough to show that for every $x\in M^*$, there is such $p_x\in H_{\theta^*}$. 
Thus fix $x\in M^*$. We claim that $p$ is such a witness. The first item is clear by \cref{reflection C existence} and that
the second one is trivial.
To see the third one holds true, fix $i<m$ and observe that
\begin{itemize}
\item $\{{\rm ht}(u): u\in x\cap T_{< \eta^*_{i}}\}$ is bounded below $\eta_i$ (as the cofinality of $\eta^*_i$ is uncountable, $x$ is countable and $M\cap\eta^*_i=M\cap \eta_i$), and
\item  there is no node in $\dom(f_p)$  whose height lies in the interval  $\big[{\rm ht}(\hat{t}_i),\eta_i\big)$, (by the construction of $\sigma(t_i)$, see \cref{2-def-C-sigma} of \cref{def-C-sigma}.)
\end{itemize}
Thus if $s\in{\rm dom}(f_p)$ is of height at least ${\rm ht}(\hat{t}_i)$,  then ${\rm ht}(s)\geq\eta_i$, and thus 
\[
{\rm sup}\{{\rm ht}(u): u\in x\cap T_{< \eta^*_{i}}\}< \eta_i\leq {\rm ht}(s).
\]
\end{proof}

Fix a $T$-assignment $x\mapsto p_x$ in $M^*$. We shall show that there is a set $B^*\in M^*$ cofinal in $\mathcal P_{\omega_1}(T)$ such that
for every $x\in M^*\cap B^*$, $p_x$ and $p$ are compatible.
Let $n\coloneqq |{\rm dom}(f_p)|$.
For each $x\in \mathcal P_{\omega_1}(T)$, fix an enumeration of ${\rm dom}(f_{p_x})$, say 
$\langle t^x_j: j< n\rangle$.
For every $B\subseteq \mathcal P_{\omega_1}(T)$, let $$B(i,j)\coloneqq \{x\in B: {\rm ht}(t^x_j)\geq {\rm ht}(\hat{t}_i)\}.$$ 

Note that if $B\in M^*$, then $B(i,j)\in M^*$.

\begin{claim}\label{claim1-prop P}
Let $i< m$ and $j< n$. Suppose that $B\in M^*$ an unbounded subset of
$\mathcal P_{\omega_1}(T)$. Assume that $B(i,j)$ is cofinal in $B$.
Then, there is a cofinal subset $B_{i,j}$ of $B(i,j)$ in $M^*$ such that for every $x\in M^*\cap B_{i,j}$, $t^x_j\nless_T O_M(t_i)$.
\end{claim}
\begin{proof}
 Let $\Psi_i$ be the characteristic function of $b_M(t_i)$ on $T$.
 Note that $\Psi_i$ is not guessed in $M$.
 For every $x\subseteq T$, we let   $\psi^x_j:x\rightarrow 2$  be   defined by $\psi^x_j(s)=1$ if and only if $s<_T t^x_j$. Now consider the mapping $x\mapsto \psi^x_j$.
Since $\Psi_i$ is not guessed in $M$, \cref{Baumgartner} implies that there is a set $B_{i,j}\in M^*$ cofinal in $B(i,j)$ such that for every $x\in B_{i,j}$, $\psi^x_j\nsubseteq \Psi_i$. 

Assume towards a contradiction that there is  $x\in M^*\cap B_{i,j}$ with $t^x_j<_T O_M(t_i)$. Then $t^x_j\in M\cap T_{<\eta_i}$, and for every
$s\in x$  of height at least $\eta^*_i$, we have $\psi^x_j(s)=0=\Psi_i(s)$. 
Thus $\psi^x_j\nsubseteq \Psi_i$ implies that there is some $s\in T_{<\eta^*_i}\cap M$ such that $\psi^x_j(s)\neq \Psi_i(s)$.
Since $x\in B(i,j)$, we have ${\rm ht}(t^x_j)\in \big[{\rm ht}(\hat{t}_i),\eta^*_i\big)$. On the other hand, by \cref{3-proper} in the definition of a $T$-assignment, we have ${\rm ht}(s)<{\rm ht}(t^x_j)$. Thus $s<_T t^x_j$ if and only if $s\nless_T O_M(t_i)$, which contradicts $t^x_j<_T O_M(t_i)$.
\end{proof}

Returning to our main proof, let $e$ be a bijection between  $mn$ and $m\times n$. For every $k<mn$, set 
$e(k)\coloneqq(e_0(k),e_1(k))$. We build a descending sequence $\langle B_k:-1\!\leq k<mn\rangle$ of cofinal subsets of $ P_{\omega_1}(T)$ with $B_k\in M^*$ as follows. Let also $B_{-1}\coloneqq\mathcal P_{\omega_1}(T)$. Suppose that $B_k$, for $k\geq -1$, is constructed. Set $C^k\coloneqq B_k(e_0(k),e_1(k))$ and ask the following question:
\begin{itemize}
\item Is $C^k$ cofinal in $B_k$?
\end{itemize}

Then proceed as follows:
\begin{itemize}
\item If the answer to the above question is YES, then apply \cref{claim1-prop P} to $C^k$, $e_0(k+1)$ and $e_1(k+1)$  to obtain $ C^k_{e_0(k+1),e_1(k+1)}\in M^*$ as in the claim, and then set $B_{k+1}\coloneqq C^k_{e_0(k+1),e_1(k+1)}$.

\item If the answer to the above question  is NO, then let $B_{k+1}=B_k\setminus C^k$.
\end{itemize}

It is clear that $\langle B_k: -1\!\leq k< mn\rangle$ is descending and each $B_k$ is in $M^*$. Set $B^*\coloneqq B_{mn-1}$. 
Note that if $x\in C^k_{e_0(k+1),e_1(k+1)}$, then $t^x_{e_1(k+1)}\nless_T O_M(t_{e_0(k+1)})$, by \cref{claim1-prop P}.

\begin{claim}
For every $x\in B^*\cap M^*$, $p_x$ and $p$ are compatible.
\end{claim}
\begin{proof}
Fix $x\in B^*\cap M^*$. Then $p_x\in M^*\cap D$. Let $r=p_x\land p$. We claim that $r$ is a condition.
By \cref{compatibility lemma 2-C}, we only need to check if there are comparable
$s\in {\rm dom}(f_{p_x})\setminus {\rm dom}(f_p)$
and $t\in \mathscr{O}(p,M)$ such that $f_{p_x}(s)=f_p(t)$. 
We shall see that it  does not happen. Thus assume towards a contradiction that there are such  $t$ and $s$. Then $t=t_i$ and $s=t^x_j$, for some $i< m$ and $j< n$. 
Note that $f_{p_x}(s),t^x_j\in M$, as $x\in M^*$.  Observe that if $t_i\leq_T t^x_j$, then $t_i$ is guessed in $M$, and hence it belongs to $M$ by \cref{4} of \cref{main forcing 2}, which is a contradiction. Thus $t^x_j<_T t_i$, which in turn implies that $t^x_j\in b_M(t_i)$ (recall that $O_M(t_i)$ is not guessed in $M$.)
 Since $f_{p_x}(s)=f_p(t)$ and  $p_x\in R_p(M,\sigma)$,   \cref{p3-C-Reflection} in \cref{M-ref C def} implies that ${\rm ht}(t^x_j)\nless {\rm ht}(\hat{t}_i)$. Thus ${\rm ht}(t^x_j)\geq {\rm ht}(\hat{t}_i)$.
Let $k\geq 0$ be such that $e(k)=(i,j)$.
Since $x\in B^*\subseteq B_k\subseteq B_{k-1}$ and that ${\rm ht}(t^x_j)\geq {\rm ht}(\hat{t}_i)$, we  have   
$B_k=C^{k-1}_{i,j}$, but then  $t^x_j\nless_T O_M(t_i)$ by \cref{claim1-prop P}, which is a contradiction since $t^x_j\in b_M(t_i)$  implies that $t^x_j<_T O_M(t_i)$.
\end{proof}

\end{proof}

\begin{remark}
Note that  to find the cofinal set $B^*$ in the above proof, we could start with any set which is cofinal in $\mathcal P_{\omega_1}(T)$.
\end{remark}

\begin{corollary}\label{properness-P}
$\mathbb P_T$ is proper.
\end{corollary}
\begin{proof}
Let $\theta^*$ be a sufficiently large regular cardinal.
Assume that $M^*\prec H_{\theta^*}$ is countable and contains $H_\theta,T, \mathcal E^0$ and $\mathcal E^1$. Set $M=M^*\cap H_{\theta}$, and let $p\in M^*$ be a condition. 
Notice that the set of such models is a club in $\mathcal P_{\omega_1}(H_{\theta^*})$.
By \cref{Mtop-P}, 
$p^M$ is a condition with $p^M\leq p$ such that $M\in\mathcal M_{p^M}$.
Now, \cref{genericity-C} guarantees that $p^M$ is $(M^*,\mathbb P_T)$-generic. Thus $\mathbb P_T$ is proper.
\end{proof}

We shall use  the above strategy and \cref{Baumgartner} to show that $\mathbb P_T$ has the $\omega_1$-approximation property.

\begin{proposition}\label{approx}
$\mathbb P_T$ has the $\omega_1$-approximation property.
\end{proposition}
\begin{proof}
Assume towards a contradiction that $\dot A$ is a $\mathbb P_T$-name such that for some $p\in\mathbb P_T$ and some $X\in V$, we have
\begin{itemize}
\item  $p\Vdash ``\dot{A}\subseteq \check{X}"$,
\item $p\Vdash ``\dot{A}\notin V"$, and
\item $p\Vdash``\dot{A} \mbox{ is countable approximated in } V"$, i.e., for every countable set $a\in V$, $p\Vdash ``\dot{A}\cap\check{a}\in V"$.
\end{itemize}

 Without loss of generality, we may work with a $\mathbb P_T$-name for the characteristic function of $\dot{A}$, say $\dot{f}$. We may also, without loss of generality, assume that either $T\subseteq X$ or $X\subseteq T$. To see this, observe that by passing to an isomorphic copy of $T$, we may assume that the underlying set of $T$ is $|T|$.
 On the other hand, using a bijection between  $X$ and $|X|$, we can assume that the domain of $\dot f$ is forced to be $|X|$.
 As $|X|$ and $|T|$ are comparable, we may assume that either $T\subseteq X$ or $X\subseteq T$. 
 
Let us assume that $T\subseteq X$, the other case is proved similarly.
Let $\theta^*$ be a sufficiently large regular cardinal. Let $M^*\prec H_{\theta^*}$ be a countable model containing all the relevant objects, including $p$.  Set $M=M^*\cap H_\theta$.
We can extend $ p^M$ to a condition $q$ such that $q$ decides
$\dot{f}\rest_{M^*}$, i.e., for some  function $g:M^*\cap X\rightarrow 2$ in $V$,  $q\Vdash ``\dot{f}\rest_{M^*}=\check{g}"$.
\begin{claim}
$g$ is not guessed in $M^*$.
\end{claim} 
\begin{proof}
Suppose that $g$ is guessed in $M^*$. Let $g^*\in M^*$ be such that $g^*\cap M^*=g$. Set
\[
D=\{r\leq p: \exists x\in X ~~r\Vdash ``g^*(x)\neq \dot{f}(x)"\}\cup\{r\in \mathbb P_T: r\perp p\}.
\]
Obviously $D\in M^*$. We use elementarity to show that $D$ is dense in $\mathbb P_T$. Thus let $r\in M^*\cap \mathbb P_T$.
We may assume that $r$ is compatible with $p$. Thus, there is $s\in M\cap\mathbb P_T$ such that $s\leq p,r$.
Since $p\forces ``\dot{f}\notin V"$, there is $x\in M^*\cap X$ and there is $s'\leq s$ in $M^*$ such that $s'\Vdash ``g^*(x)\neq\dot{f}(x)"$. Thus $s'\in D\cap M$.

On the other hand, by \cref{genericity-C}, $q$ is $(M^*,\mathbb P_T)$-generic. Thus, there is $u\in D\cap M^*$ such that $u||q$. But then $u||p$, and thus
there is $x\in M^*\cap X$  such that $u\Vdash ``g^*(x)\neq\dot{f}(x)"$. This is impossible, as $q\Vdash g^*(x)=g(x)=\dot{f}(x)$.

\end{proof}

Fix an $M$-support set $\sigma$ for $q$.
As in the proof of \cref{genericity-C}, we can find, in $M^*$, 
  a function $x\mapsto (q_x,g_x)$   on $\mathcal P_{\omega_1}(X)$ such that:
 \begin{enumerate}
\item $q_x\in R_p(M,\sigma)$
\item  $ |\dom(f_{q_x})|=|\dom(f_q)|$.
\item For every $s\in{\rm dom}(f_{q_x})$  and every $i< m$, if ${\rm ht}(s)\in\big[{\rm ht}(\hat{t}_i),\eta^*_i\big)$, then 
    $${\rm sup}\{{\rm ht}(u): u\in x\cap T_{< \eta^*_{i}}\}<{\rm ht}(s).$$
 
     \item $g_x:{\rm dom}(g_x)\rightarrow 2$ is a function with countable domain containing $x$ as a subset. 
     \item $q_x\Vdash g_x\rest_x=\dot{f}\rest_x$.
 \end{enumerate}
 Here, $\eta_i$,  $\eta^*_i$ and $\hat{t}_i$ are  as in the proof of \cref{genericity-C}. Note that to find an assignment in $M^*$, observe that if $x\in M^*$, then $x\subseteq \dom(g)$, and thus we can use $(q,g)$ as a witness.
Since, we assumed $T\subseteq X$ and by the above claim $g$ is not guessed in $M^*$, we first apply \cref{Baumgartner} to find a set $B\in M^*$, cofinal in $\mathcal P_{\omega_1}(X)$, such that for every $x\in B$, $g_x\nsubseteq g$.
 Now let $C$ be the restriction of $B$ to $T$, i.e.,
 $C=\{x\cap T:x\in B\}$. Then $C$ is cofinal in $\mathcal P_{\omega_1}(T)$. Using the Axiom of Choice, for each $c\in C$, pick  $x_c\in B$  such that $x_c\cap T=c$. Fix such a choice function $c\mapsto x_c$ in $M^*$ and
 consider the assignment $c\mapsto q_{x_c}$.
By the above properties, $c\mapsto q_c=q_{x_c}$ is a $T$-assignment in $M^*$. Thus, as in \cref{genericity-C}, there is some $c\in C\cap M^*$ such that $q_{c}$ is compatible with $q$. There exists $x\in B\cap M^*$ with $x_c=c$, but this is a contradiction, as $g_x\nsubseteq g$ implies that $q_{x_c}=q_c$ is not compatible with $q$!

\end{proof}

\begin{lemma}\label{domain}
Suppose that $p\in \mathbb P_T$ and $t\in T$. 
Then there is some $q\leq p$ such that $t\in {\rm dom}(f_q)$.
\end{lemma}
\begin{proof}
Assume that $t$ is not in ${\rm dom }(f_p)$. If $t$  is not in any  model belonging to $\mathcal E^0_p$, then pick $\nu$ below $\omega_1$ and different from the values of $f_p$ such that
 \[
 \nu>{\rm max}\{M\cap\omega_1: M\in\mathcal{E}^0_p\},
 \]
and  then set  $q=(\mathcal M_p, f_p\cup\{(t,\nu)\})$. Then \cref{1} of \cref{main forcing 2} is easily fulfilled, \cref{2} holds true as $\nu\notin{\rm rang}(f_p)$.
\cref{closed under f} is obvious as $t$ does not belong to any model in $\mathcal M_q=\mathcal M_p$. Finally,  \cref{4}  is fulfilled, since  $f_q(t)=\nu$ belongs to no  model in $\mathcal E^0_q=\mathcal E^0_p$.

Now assume that there are some models in $\mathcal E^0_p$ containing $t$.
Let $M$ be the least countable model in $\mathcal M_p$ with $t\in M$.
Let $\nu\in M\cap\omega_1\setminus {\rm ran}(f_p)$ be such that $$\nu>{\rm max}\{N\cap\omega_1: N\in\mathcal{E}^0_p\cap M\}.$$
Set $q=(\mathcal M_,f_p\cup\{(t,\nu)\})$.
We claim that $q$ is a condition.
As in the previous case, \cref{1,2} of \cref{main forcing 2} hold true, thus
we only need to check \cref{closed under f,4}.\\

\textbf{\cref{closed under f}:} Assume that $N\in\mathcal{E}^0_p$ contains
$t$. By the minimality of $M$, $M\in^* N$. We claim that $M\subseteq N$.
Suppose this is not the case. Thus there is some $P\in\mathcal{E}^1_p$ such that
$P\cap N\in^* M\in P\in N$, but then $t\in P\cap N$, which contradicts the minimality of $M$. Thus $M\subseteq N$, and hence  $\nu\in M\subseteq N$.\\

\textbf{\cref{4}:} Suppose that $N\in\mathcal{E}^0_p$ is such that $\nu\in N$ and $t$ is guessed in $N$.
We shall show that $M\subseteq N$, and hence $t\in N$.
We first show that  $N\in^* M$ is impossible. To see this, observe that   $N\notin M$ by our choice of  $\nu$.
Thus  if $N\in^* M$, then there is some $P\in\mathcal{E}^1_p\cap M$ such that $N\in [P\cap M,P)_p$. Now
$t$ belongs to $P$ as it is guessed in $N\subseteq P$, and thus $t\in P\cap M$, which contradicts the minimality of $M$.

Now if $M\nsubseteq N$, there is $P\in\mathcal M_p$ such that $P\cap N\in^* M\in P\in N$.
Then  since $t\in P$ is guessed in $N$, by \cref{guessing lemma}, $t$ is  guessed in $P\cap N$. Notice that $\nu\in P\cap N\in^* M $, which is a contradiction as $P\cap N\in^* M$, as is was shown in the previous paragraph.
\end{proof}

\begin{remark}
Notice that $\mathbb P_T$
 forces $|H_{\theta}|=|T|=\aleph_2$.
 \end{remark}

\section{Conclusion}\label{sec5}
In this section, we prove our main theorem.

\begin{theorem}\label{main theorem}
Assume that  ${\rm GM}^*(\omega_2)$ holds. Then, every tree of height $\omega_2$ without cofinal branches is specialisable via a proper and $\aleph_2$-preserving forcing with finite conditions. Moreover, the forcing has the $\omega_1$-approximation property.
\end{theorem}
\begin{proof}
By \cref{tree iso}, we may also assume  that $T$ is a Hausdorff  tree.
By \cref{properness-P,aleph2 preserv}, $\mathbb P_T$ preserves $\aleph_1$ and $\aleph_2$, respectively. Let
$G\subseteq\mathbb P_T$ be $V$-generic filter, and set
$$f_G=\bigcup\{f_p:p\in G\}.$$
By
\cref{domain}, $f_G:T\rightarrow \omega_1$ is a total function on $T$. It is clear that $f_G$ is a specialising function on $T$.
\end{proof}

Since $\rm PFA$ implies ${\rm GM}^*(\omega_2)$ by \cref{fact Viale-Weiss}, we obtain the following corollary.
\begin{corollary}
Assume $\rm PFA$. Suppose $T$ is a tree of height $\omega_2$ without cofinal branches. Then there is a proper and $\aleph_2$-preserving forcing with the $\omega_1$-approximation property such that $T$ is special in generic extensions by
$\mathbb P_T$.
\end{corollary}
\begin{proof}[\nopunct]

\end{proof}

\bigskip
\footnotesize
\noindent\textit{Acknowledgements.}
The author's research was supported through the project M 3024 by the Austrian Science Fund (FWF). The author is grateful to M. Golshani and B. Veličković for the fruitful conversations about the contents of this paper. The author would like to extend his thanks to the referee for their careful reading and constructive comments that significantly improved the readability of this manuscript.

\normalsize
\baselineskip=17pt

\printbibliography[heading=bibintoc]

\end{document}